\documentclass[notitlepage,11pt]{article}

\usepackage{akshay}
\usepackage{mathrsfs}
\usepackage{setspace}

\DeclareMathOperator{\persp}{\mathcal{P}}

\newcommand{\strict}[1]{\operatorname{strict} #1}
\newcommand{\slack}[1]{\operatorname{slack} #1}
\newcommand{\subspace}[1]{\operatorname{sub} #1}

\newcommand{\E}{\euclid}
\newcommand{\EE}{\euclid^{\prime}}

\newcommand{\K}{K}
\newcommand{\C}{\mathscr{C}}
\newcommand{\CC}{\C^{\prime}}
\newcommand{\conel}[1][\K]{\prec_{#1}}
\newcommand{\coneleq}[1][\K]{\preceq_{#1}}
\newcommand{\coneg}[1][\K]{\succ_{#1}}
\newcommand{\conegeq}[1][\K]{\succeq_{#1}}

\newcommand{\X}{S}
\newcommand{\XX}{X^{\prime}}

\newcommand{\opt}[1]{z^{\ast}_{#1}}
\newcommand{\cstar}{\opt{P}}
\newcommand{\optdual}{\opt{D}}
\newcommand{\dualfeas}{\mathcal{C}_{d}}%{\mapAad(\dualcone{\K}) - \dualcone{C}}
\newcommand{\primalfeas}{\mathcal{C}_{p}}%{\A(C) + \K}
\newcommand{\mapA}[1][A]{\mathcal{#1}}
\newcommand{\A}[1][]{\mathcal{A}#1}

\renewcommand{\L}{\mapA[L]}
\newcommand{\G}{\mapA[G]}
\newcommand{\Lp}{\mathcal{L}_{p}}
\newcommand{\Ld}{\mathcal{L}_{d}}
\newcommand{\mapAad}[1][]{\adjoint{\A}#1}
\newcommand{\adjLp}{\mathcal{L}^{\star}_{p}}
\newcommand{\adjLd}{\mathcal{L}^{\star}_{d}}

\title{On strong duality, theorems of the alternative, and projections in conic optimization}

\author{Temitayo Ajayi\affilinfo{Department of Computational and Applied Mathematics, Rice University, USA} 
\and Akshay Gupte\affilinfo{School of Mathematics, University of Edinburgh, UK}[akshay.gupte@ed.ac.uk] 
\and Amin Khademi\affilinfo{Department of Industrial Engineering, Clemson University, USA} 
\and Andrew J. Schaefer\affilinfo{1}}

\date{November 10, 2021}

\begin{document}

\maketitle

\begin{abstract}
A conic program is the problem of optimizing a linear function over a closed convex cone intersected with an affine preimage of another cone. We analyse three constraint qualifications, namely a Closedness CQ, Slater CQ, and Boundedness CQ (also called Clark-Duffin theorem), that are sufficient for achieving strong duality and show that the first implies the second which implies the third, and also give a more general form of the third CQ for conic problems. Furthermore, two consequences of strong duality are presented, the first being a theorem of the alternative on almost feasibility (also called weak infeasibility), and the second being an explicit description of the projection of conic sets onto linear subspaces, akin to using projection cones for polyhedral sets.

\mykeywords{Closedness of adjoint image}{Slater constraint qualification}{Clark-Duffin theorem}[Theorem of Alternative][Projection onto subspace][Bounded feasible set]

\mysubjclass{90C46, 90C25, 49N15, 90C22}
\end{abstract}

\section{Introduction}

Let $C\subseteq\E$ and $\K\subseteq\EE$ be nonempty closed convex cones in Euclidean spaces $\E$ and $\EE$ having respective inner-products $\inner{\cdot}{\cdot}[\E]$ and $\inner{\cdot}{\cdot}[\EE]$. For a linear map $\A\colon\E\to\EE$ and vectors $b\in\EE$ and $c\in\E$, the conic optimization problem is
\begin{subequations}
\begin{equation}	\label{def:primal}
\cstar \eq \sup\left\{\inner{c}{x}[\E]\sep \A[x] \coneleq b, \ x\in C \right\}.%x\in X\right\}, \quad X \define \left\{x\in C \sep \A[x] \coneleq b \right\}.
\end{equation}
where the constraints $\A[x] \coneleq b$ mean that $b - \A[x]\in\K$.  Because $\K$ can be a Cartesian product of finitely many cones, the conic inequalities $\A[x] \coneleq b$ can incorporate multiple constraints over different cones. Equality constraints can be represented as conic constraints with respect to the singleton cone $\{\zeros\}$. 
The conic constraint $x\in C$ is kept separate in \eqref{def:primal} because it involves a special linear map (identity map) and constant vector being all zeros. For nontriviality, we assume that $C\neq\{\zeros\}$ and %$\K\neq\EE$; the latter ensures that the primal problem is not linear optimization over a cone, which has value either 0 or $\infty$. 
that at least one of $C$ or $\K$ is not equal to its ambient Euclidean space. The Lagrangian dual problem is the conic program
\begin{equation}		\label{def:dual}
\optdual \eq \inf\left\{\inner{b}{y}[\EE] \sep \mapAad[y] \conegeq[\dualcone{C}] c, \ y \in \dualcone{\K} \right\}.
\end{equation}
\end{subequations}
for the adjoint linear map $\mapAad\colon\EE\to\E$ and dual cones $\dualcone{\K}\subseteq\EE$ and $\dualcone{C}\subset\E$. %This problem is commonly also referred to as the conic dual of \eqref{def:primal} because it is also a conic program and Lagrangian duals of general convex programs do not always have their min-max simplified into an explicit minimization problem. 
The primal-dual pair are symmetric since the conic dual of \eqref{def:dual} yields \eqref{def:primal}.

Duality gap is the nonnegative difference $\optdual - \cstar$. Strong duality means that $\cstar = \optdual$ (allowing for values $\pm \infty$), and if these values are finite then at least one of the two problems is solvable (i.e., has an optimal solution). When both $\K$ and $C$ are polyhedral, feasibility of one of the problems guarantees strong duality, but this is not true for general $\K$ and $C$. There is abundant literature on sufficient conditions for strong duality in conic programs, see  \citep{luo1997duality,nesterov1992conic,bentalnem2001,tunccel2012strong,andersen2002notes,pataki2013strong,Ramana_1997,ramana97sdp,Shapiro_2001,schurr2007universal,zualinescu2008zero,kostyukova2021strong,glineur2001proving}. The most well-known of these conditions is the generalised Slater constraint qualification (CQ) which requires that one of the two problems be strictly feasible (in the sense of satisfying the conic constraints through their relative interiors). There are many proofs of sufficiency of Slater CQ, and it also follows from the classical Fenchel duality theorems \citep[chap.~31]{rockafellar1970convex}. There have been studies on the geometry of Slater CQ and degeneracy of solutions \cite{Drusvyatskiy_2017,Pataki_2001,dur2017genericity}. Primal problem~\eqref{def:primal} can be extended to what is called an abstract convex program in Banach spaces by considering a $\K$-convex function $g(x)$ and replacing $b-\A[x]\in\K$ with $-g(x)\in\K$; strong duality results for these can be found in \citep{Borwein_1981,ban2009duality,jeyakumar1990zero,jeyakumar2008complete,boct2007new,boct2005strong}.

In this paper, we make two observations about sufficient conditions for strong duality between~\eqref{def:primal} and \eqref{def:dual}. We add to the vast literature on Slater CQ by showing that it follows from a more general Closedness CQ shown by \citep{Shapiro_2001,barvinok2002course} but which is probably not as well-known. Our argument relies on known sufficient conditions for closedness of linear image of a closed convex cone and noting a property about perspective images of a convex cone. We also give another proof for the sufficiency of the Boundedness CQ, which has been called the Clark-Duffin theorem \citep{duffin1978clark} for convex programs with inequality constraints. A restricted version of this with $\K=\{\zeros\}$ and $C$ being a proper cone appeared in \citep[Theorem 3.15]{schurr2007universal} in the context of universal duality. Our derivation relies on a conic version of the Gordan's theorem of the alternative. Given the importance of these two CQs, we also present several necessary/sufficient conditions for them to hold.

Another set of contributions are obtaining consequences of strong duality through Slater CQ. We give a theorem of the alternative (\mythref{polaralmost}) for almost feasibility of a problem, and give an explicit description of the projection of a conic set onto a linear subspace (\mythref{projection}). The former generalises \citep[Proposition 1.7.1]{bentalnem2001} which deals with a proper cone $\K$ and $C=\E$, and the latter extends the notion of projection cones for polyhedral sets.

\textsection\ref{sec:prelim} specifies the notation used in this paper, basic definitions and some fundamental results from convex analysis. \textsection\ref{sec:basic} gives some detailed observations about conic optimization, in particular, a basis representation of the dual that implies any of the duality results can be translated to be in terms of the span of $C$, and characterisations of the recession cone of $X$ and its polar cone. The three sufficient conditions for strong duality --- Closedness CQ, Slater CQ and Boundedness CQ, are analysed in \textsection\ref{sec:closed} for their inter-relationships and conditions for them to hold. \textsection\ref{sec:almostfeas} gives a new theorem of the alternative and \textsection\ref{sec:proj} projection of conic sets. Primal-dual symmetry means that results for one problem can be directly extended to the other problem, and so we generally prove our results only for the primal and omit analogous statements/proofs for the dual.

\section{Preliminaries} \label{sec:prelim} %Elementary Results on Duality

%We begin by describing general notation that will be used in this paper, followed by some useful properties of convex sets and cones. 

\subsection{General Notation \& Terminology}	\label{sec:notation}

A Euclidean space is a finite-dimensional inner-product space over the reals.  
Unless there is ambiguity, we drop the subscripts in the notation for inner-products. %Elementwise comparison between two vectors $x,y\in\E$ and other ambient spaces is written as $x\ge y$ or $x\le y$. 
The vector of all zeros is written as $\zeros$. The Minkowski sum of two sets $\X_{1},\X_{2}\subset\E$ is $\X_{1}+\X_{2}$, their Cartesian product (equiv. direct sum) is  $\X_{1}\times\X_{2}$, and we write the Minkowski difference as $\X_{1} - \X_{2}$ and define it as the Minkowski sum $\X_{1} + (-\X_{2})$. %Note that Minkowski addition is the image of a linear map from the product space $\E\times\E$ to $\E$. 
An ordered pair of $x\in\E$ and $y\in\EE$ is the tuple $(x,y)$ which belongs to the product space $\E\times\EE$. 
For a set $\X\subset\E$, $\relint{\X}$ is the relative interior, $\intr{\X}$ is the interior, $\closure{\X}$ is the closure, $\boundr{\X}:=\closure{\X}\setminus\relint{\X}$ is the relative boundary, $\aff{\X}$ is the affine hull, $\subspace{\X}$ is the linear subspace parallel to $\aff{\X}$ (equal to $\aff{\X} - x$ for any $x\in\X$), and $\linhull{\X}$ is the linear hull (span). %Note that $\aff{\X} = \subspace{\X} = \linhull{\X}$ when $\zeros\in\X$. 
The orthogonal complement of a linear subspace $L\subset\E$, also called its annihilator, is denoted by $L^{\bot}$. %, and note that $(L^{\bot})^{\bot} = L$. Two linear subspaces $L_{1},L_{2}\subset\E$ have $L_{1}\subseteq L_{2}$ if and only if $L_{1}^{\bot}\supseteq L_{2}^{\bot}$. 
%For a set $\X$, we denote $\X^{\bot} := (\subspace{\X})^{\bot}$. 
The recession cone of a nonempty convex set $\X\subset\E$ is the convex cone $\rec{\X} := \{r\in\E \colon x + \mu r\in\X, \, \forall x\in\X, \mu\ge 0\}$; when $\X$ is also closed, it is known that $\rec{X} =  \{r\in\E \colon \exists\,x\in\X \text{ s.t. } x + \mu r\in\X, \, \forall \mu\ge 0\}$. %The relative interior and the closure have the same recession cone as the set. 
The lineality space of $\X$ is defined as $\lineal{\X} := \rec{\X}\cap-\rec{\X}$, and when this is trivial (equal to $\{\zeros\}$) the set is called pointed. 

The polar cone and the dual cone of a convex cone $\C$ are denoted by $\polar{\C}$ and $\dualcone{\C}$. For a convex cone $\C$ containing $\zeros$, the binary relation $\coneleq[\C]$ is a quasi-ordering which becomes a partial order when $\C$ is pointed. It is additive equivariant and distributes over the Cartesian product of cones. 
%reflexive, transitive, additive equivariant ($u\coneleq[\C] v, x\coneleq[\C] y$ implies $u+x\coneleq[\C] v+y$), 
%and distributes over the Cartesian product of cones, i.e., $(x,u) \coneleq[\C\times\CC] (y,v)$ if and only if $x\coneleq[\C]y$ and $u\coneleq[\CC]v$. It is also antisymmetric when the cone is pointed, thereby making $\coneleq[\C]$ a partial order. 
Strict inequality for the binary relation $\coneleq[\C]$ is written as $x\conel[\C] y$ and is defined as $y-x\in\relint{\C}$. 

For a map $\L\colon\E\to\EE$, the image of a set $\X\subset\E$ is $\L(\X) := \{\L(x)\colon x\in\X \}$, the image of $\L$ is $\range{\L} := \L(\E)$, the preimage of a set $\X\subset\EE$ is $\L^{-1}(\X) := \{x\in\E\sep \L(x)\in\X \}$. %which is a union of affine subspaces as $\cup_{y\in\X}\,\mapA[G]^{-1}(\{y\})$. 
We work with linear and affine maps only. The kernel of a linear map $\L$ is $\kernel\L := \L^{-1}\{\zeros\}$.  
%in particular, the translate $x\mapsto\G(\zeros)-\G(x)$ of an affine map $\G$ is a linear map, and if $\L$ is a linear map then for any $v\in\E$ the map $x\mapsto v + \L(x)$ is an affine map. 
The adjoint of a linear map $\L$ is a unique linear map $\adjoint{\L}\colon \EE\to\E$
%\mapsto(\inner{G_{1}}{y},\inner{G_{2}}{y},\dots,\inner{G_{n}}{y})\in\E$ 
which satisfies $\inner{\L(x)}{y}[\EE] = \inner{\adjoint{\L}(y)}{x}[\E]$ for all $x\in\E,y\in\EE$. %, and hence $\adjoint{(\adjoint{\L})}(\cdot) = \L(\cdot)$. 
%We assume familiarity with fundamental properties of linear operators and their adjoints \citep[cf.][]{roman2008}.
%For a linear map $\L$ and linear subspace $L\subset\EE$, the preimage $\L^{-1}(L)$ is a linear subspace in $\E$. 
Affine maps are translates of linear maps. For every affine map $\G$ we have the associated linear map
\[
\L_{\G}(x) \define \G(\zeros) - \G(x). 
\]

%The dual cone and polar cone are denoted by $\dualcone{\C}$ and $\polar{\C}$, respectively, and defined as $\dualcone{\C}:=\{y\in\real^{n}\colon \inf\{y^{\top}x\colon x\in\C \} = 0\}$ and $\polar{\C}:=\{y\in\real^{n}\colon \sup\{y^{\top}x\colon x\in\C \} = 0\} = -\dualcone{\C}$. These are closed convex cones. 

% and $\polar{\C} = \{y\in\E\colon \sigma_{\C}(y) \le 0\}$. 
%A cone $\C\subseteq\E$ is a set that is closed under positive scaling ($x\in\C,\lambda > 0$ implies $\lambda x \in\C$). We do not require a cone to contain $\zeros$, but when the cone is closed it necessarily contains $\zeros$. A cone is nontrivial if $\C\neq\{\zeros\}$ and $C\subsetneq\E$. The affine hull and span of a cone are equal and written as $\aff{\C}$, the orthogonal complement of $\aff{\C}$ is $\C^{\bot}$. 

%Most interesting cones, such as pointed cones, have $\dualcone{\C}\neq\C^{\bot}$, which is equivalent to $\dualcone{\C}\setminus\C^{\bot}\neq\emptyset$. However, even non-pointed cones can have $\dualcone{\C}\neq\C^{\bot}$, for e.g., $\C = \real_{+}\times\real\times\{0\}$ has $\dualcone{\C} = \real_{+}\times\{0\}\times\real$ and $\C^{\bot} = \{(0,0)\}\times\real$.  

The feasible sets of the primal~\eqref{def:primal} and dual~\eqref{def:dual} are denoted by
\begin{subequations}	\label{def:XY}
\begin{align}
X \eq X(b) &\define \left\{x \in C\sep \A[x] \coneleq b \right\} \eq \left\{x \in C\sep \exists\, s\in\K \text{ s.t. } s = b-\A[x] \right\}, \\
Y \eq Y(c) &\define \left\{y \in \dualcone{\K} \sep \mapAad[y] \conegeq[\dualcone{C}] c \right\} \eq \left\{y \in \dualcone{\K} \sep \exists\, w\in\dualcone{C} \text{ s.t. } w = \mapAad[y] - c \right\}.
\end{align}
%By allowing the right-hand sides to vary, we get the parametric conic set ***
%Define the parametric conic set as
%\begin{equation}	\label{parametric}
%X(\beta) \define \{x\in C\colon \A[x]\coneleq\beta\}, \quad \beta\in\EE,
%\end{equation}
%and note that $X(\beta)$ could be empty for some right-hand sides $\beta$.
These are parametrised by their respective right-hand sides because it will be sometimes be necessary to refer to them in the parametric form. We will use the parametric and non-parametric forms as appropriate depending on the context. 
%As convenient and depending on context, we may refer to these either as parametrised sets 
The feasible sets can also be expressed as affine preimages of closed convex sets,
\begin{equation}	\label{feaspreimg}
X \eq G_{p}^{-1}(\K\times C), \qquad Y \eq G_{d}^{-1}(\dualcone{C}\times\dualcone{\K}).
\end{equation}
where $G_{p}\colon x\mapsto (b-\A[x],x)$ and $G_{d}\colon y\mapsto(\mapAad[y] -c,y)$. The adjoints of the corresponding  linear maps $\L_{G_{p}} = (\A[x],-x)$ and $\L_{G_{d}} = (-\mapAad[y],-y)$ are
\begin{equation}	\label{adjL}
\adjoint{\L}_{G_{p}}(y,w) \eq \mapAad[y] - w, \qquad \adjoint{\L}_{G_{d}}(x,s) \eq -\A[x] - s.
\end{equation}
\end{subequations}
The sets of feasible right-hand sides for the primal and dual are represented by the following two convex cones
\begin{equation}	\label{feascones}
\primalfeas \define \A(C) + \K \eq - \adjoint{\L}_{G_{d}}(C\times\K), \qquad \dualfeas \define \mapAad(\dualcone{\K}) - \dualcone{C} \eq \adjoint{\L}_{G_{p}}(\dualcone{\K}\times\dualcone{C}).
\end{equation}

\begin{observation}		\thlabel{genfeas}
$\primalfeas = \{b \in \EE \colon X(b) \neq \emptyset \}$ and $\relint{\primalfeas} = \{b\in\EE\colon \exists\, x \in \relint{C} \text{ s.t. } \A[x] \conel b \}$. Similarly, $\dualfeas = \{c\in\E\colon Y(c) \neq \emptyset \}$ and $\relint{\dualfeas} = \{c\in\E\colon \exists\, y\in\relint{\dualcone{\K}} \text{ s.t. } \mapAad[y] \coneg[\dualcone{C}] c \}$.
\end{observation}

For completeness, a proof of this is given in Appendix~\ref{sec:missing}.
%Since we could not find references in literature for this observation, elementary arguments are given in Appendix~\ref{sec:missing} for sake of completeness.

%By convention, we say that the primal optimum $\cstar$ from \eqref{def:primal} is equal to $-\infty$ if the primal is infeasible and $\cstar = +\infty$ if the primal is unbounded, which necessarily means that the feasible set $X$ is unbounded. For the dual, infeasibility is $\optdual = +\infty$ and unboundedness is $\optdual = -\infty$. 

\subsection{Convexity}

%Wherever we were unable to find standard literature references~\cite{rockafellar1970convex,bentalnem2001,luo1997duality,Boyd2004}, proofs are given in Appendix~\ref{sec:missing} for  sake of completeness.

%Some useful technical lemmata from convex analysis are stated here.

Every nonempty convex set in $\E$ has a nonempty relative interior. The following well-known results from convex analysis about closure and relative interior will be used as standard facts throughout this paper without necessarily referencing them every time they are used.

\begin{lemma}[cf. {\citep{rockafellar1970convex}}]		\thlabel{facts}
%$\relint{}$ commutes with linear map.
Let $\X$ be a nonempty convex set and $\G$ be an affine map.
\begin{enumerate}
\item $\relint{\X} = \relint{(\closure{\X})}\subseteq\closure{(\relint{\X})} = \closure{\X}$.
\item $\relint{\G(\X)} = \G(\relint{\X})\subseteq\G(\closure{\X}) \subseteq \closure{\G(\X)}$, with $\G(\closure{\X}) = \closure{\G(\X)}$ when $\G(\closure{\X})$ is closed.
\item $\G^{-1}(\closure{\X}) \supseteq \closure{\G^{-1}(\X)}\supseteq\relint{\G^{-1}(\X)}\supseteq\G^{-1}(\relint{\X})$,  with $\relint{\G^{-1}(\X)}=\G^{-1}(\relint{\X})$ and $\G^{-1}(\closure{\X}) = \closure{\G^{-1}(\X)}$ when $\G^{-1}(\relint{\X})\neq\emptyset$.
\item For two convex sets $\X_{1},\X_{2}\neq\emptyset$, we have
\begin{enumerate}
\item $\relint{(\X_{1}\times\X_{2})} = \relint{\X_{1}}\times\relint{\X_{2}}$ and $\closure{(\X_{1}\times\X_{2})} = \closure{\X_{1}}\times\closure{\X_{2}}$.
\item if $\relint{\X_{1}}\cap\relint{\X_{2}}\neq\emptyset$, then $\relint{(\X_{1}\cap\X_{2})} = \relint{\X_{1}}\cap\relint{\X_{2}}$ and $\closure{(\X_{1}\cap\X_{2})} = \closure{\X_{1}}\cap\closure{\X_{2}}$.
\item $\relint{(\X_{1}\pm\X_{2})} = \relint{\X_{1}} \pm \relint{\X_{2}}$.
\end{enumerate}
\end{enumerate}
\end{lemma}

%Commutativity of the $\relint{}$ operator with a linear map leads to its distributivity over the Minkowski sum and difference.
%
%\begin{lemma}	\thlabel{relintsum}
%$\relint{(\X_{1}\pm\X_{2})} = \relint{\X_{1}} \pm \relint{\X_{2}}$ for nonempty convex sets $\X_{1},\X_{2}\subset\E$.
%\end{lemma}
%\begin{proof}
%The Minkowski sum is the image of the linear map $\L\colon (u,v)\mapsto u+v$, i.e., $\X_{1}+\X_{2} = \L(\X_{1}\times\X_{2})$. Hence, $\relint{(\X_{1}+\X_{2})} = \relint{\L(\X_{1}\times\X_{2})} = \L(\relint{(\X_{1}\times\X_{2})}) = \L(\relint{\X_{1}} \times \relint{\X_{2}}) = \relint{\X_{1}} + \relint{\X_{2}}$.
%\end{proof}

The Minkowski sum of closed sets is not closed in general, see \citep[chap. 9]{rockafellar1970convex} for sufficient conditions for the sum to be closed.

Now let $\C$ be a closed convex cone. Since $\zeros\in\C$, the affine hull and span of a cone are equal and written as $\linhull{\C}$, the orthogonal complement of $\linhull{\C}$ is written as $\C^{\bot}$. The following relationships can be used to obtain dual counterparts of the conditions/assertions on the primal problem.

\begin{lemma}[cf. {\cite[Corollary 1]{luo1997duality}}]	\thlabel{spandual}
$\C^{\bot} = \dualcone{\C}\cap\polar{\C}$, and $\linhull{\polar{\C}} = (\lineal{\C})^{\bot}$.
\end{lemma}

%For convex cones, duality and polarity are anti-inclusion preserving and distributive over the Cartesian product.  
%Quasi-self-dual cones are a generalization of self-dual cones which have $\C=\dualcone{\C}$. Every quasi-self-dual cone is also pointed and hence a regular cone.  
 %\citet[Theorem 1.6.1 and Corollary 1.6.1]{bentalnem2001} have more well-known relationships with the dual cone.
%For a convex cone $\C$, we have $\dualcone{(\dualcone{\C})}=\closure{\C}$, $\dualcone{(\closure{\C})} = \dualcone{(\relint{\C})}=\dualcone{\C}\supsetneq\relint{\dualcone{\C}}$, and $\C^{\bot}=\dualcone{\C}\cap\polar{\C}$. The first relationship is called the Bipolar Theorem for cones. Let us formally note the following consequence of the Bipolar Theorem that is useful in some arguments for the dual problem.
%
%\begin{lemma}[cf. {\cite[Corollary 1]{luo1997duality}}]	\thlabel{spandual}
%$\aff{\dualcone{\C}} = \aff{\polar{\C}} = (\lineal{\C})^{\bot}$.
%\end{lemma}
%We allow our cones to be equal to $\{\zeros\}$ or $\E$. 

A pointed (resp. full-dimensional) $\C$ is equivalent to $\polar{\C}$ being full-dimensional (resp. pointed). Another basic result on polarity is an identity for the polar of a linear preimage of a convex cone. This is also related to the Farkas lemma for conic linear systems \citep[cf.][Theorem 2.1]{Dinh2014farkas} which states that for any $c\in\euclid$, exactly one of the following holds: either $c\in\closure{\adjoint{\L}(\dualcone{\C})}$ or there exists $y\in\euclid$ such that $\inner{c}{y} > 0$ and $\L(y)\coneleq[\C]\zeros$.

\begin{lemma}[cf. {\citep[Lemma 3]{luo1997duality}}]	\thlabel{polarinv}
A linear map $\L$ and closed convex cone $\C$ satisfy\footnote{The first identity follows from the second identity after applying the Bipolar Theorem \citep[Theorem 14.5]{rockafellar1970convex} to the set $\L^{-1}(-\X)$ (which contains $\zeros$ due to $\zeros\in\X$) and using the fact that a convex set and its closure have the same polar. } 
\[
\polar{\left(\L^{-1}(-\C)\right)} \eq \closure{\adjoint{\L}(\dualcone{\C})}, \qquad \L^{-1}(-\C) \eq %\polar{(\closure{\adjoint{\L}(\dualcone{\X})})} \eq 
\polar{(\adjoint{\L}(\dualcone{\C}))}.
\]
\end{lemma}

Linear subspaces are closed convex cones, but we may sometimes need to exclude such pathological cones, which can be characterised as follows:
\begin{equation}	\label{subspace}
%\C \text{ is not a linear subspace} \iff \C\setminus\lineal{\C} \neq \emptyset \iff \dualcone{C} \text{ is not a linear subspace} \iff \dualcone{\C}\setminus\C^{\bot}\neq\emptyset.
\C \text{ is a linear subspace} \iff \C\eq\lineal{\C} \iff \dualcone{C} \text{ is a linear subspace} \iff \dualcone{\C}\eq\C^{\bot}.
\end{equation}
For a non-subspace cone the origin is not contained in the relative interior of the cone or that of the dual cone\footnote{This is a generalization of $\zeros\notin\intr{\dualcone{\C}}$ when $\C$ is pointed which is directly due to the fact that a closed convex set containing $\zeros$ is unbounded if and only if $\zeros$ is not in the interior of its polar set  \cite[Corollary 14.5.1]{rockafellar1970convex}}, because the boundary of the dual cone contains the complement.

\begin{lemma} \thlabel{relintprop3}
If $\C$ is not a linear subspace and not equal to $\{\zeros\}$, then 
\begin{enumerate}
\item $\C^{\bot}\subseteq\boundr{\dualcone{\C}}$ and $\lineal{\C}\subseteq\boundr{\C}$,
\item $\zeros\notin\relint{\C}\cup\relint{\dualcone{\C}}$.
\end{enumerate}
\end{lemma} 

%When $\X$ is a closed convex cone $\C$, the polar cone of $\L^{-1}(-\C)$ can also be obtained from the Farkas lemma for conic linear systems \citep[cf.][Theorem 2.1]{Dinh2014farkas} which states that for any $c\in\euclid$, exactly one of the following holds: either %there exists $x\in\dualcone{\C}$ such that $c = \adjoint{\L}(x)$ 
%$c\in\closure{\adjoint{\L}(\dualcone{\C})}$ or there exists $y\in\euclid$ such that $\inner{c}{y} > 0$ and $\L(y)\coneleq[\C]\zeros$. %This conic Farkas lemma can be seen as a consequence of the separation theorem, and so our proof  for \mythref{polarinv}, which also relies on the separation theorem, is more direct and also standalone. 
%An analogous result for subspaces is the following, which generalises the fundamental fact from linear algebra that the orthogonal complement to the kernel of a linear map is exactly the image of the adjoint map. 

%\begin{lemma}	\thlabel{linearalg}
%$(\L^{-1}(L))^{\bot} = \adjoint{\L}(L^{\bot})$ for a linear map $\L\colon\E\to\EE$ and linear subspace $L\subseteq\EE$.
%\end{lemma}

%Now we state several lemmata on cones that are useful in this paper. %some of which may be known in literature but for the sake of completeness and self-containment, we give proofs wherever we were unable to find references in standard texts such as \cite{bentalnem2001,rockafellar1970convex,Boyd2004}. 
%Throughout, let $\C\subseteq\E$ be a nonempty closed convex cone, unless stated otherwise.

We also have a conic version of Gordan's theorem of the alternative from linear programming. 

\begin{lemma}[Conic version of Gordan's theorem]	\thlabel{gordan}
Consider the following two statements:
\begin{enumerate}
\item there exists $x\in C\setminus\{\zeros\}$ such that $-\A[x] \in \K $, %\coneleq\zeros$,
\item there exists $y\in\relint{\dualcone{\K}}$ such that $\mapAad[y] \in \relint{\dualcone{C}}$. %\coneg[\dualcone{C}]
\end{enumerate}
If (1) is not true, then (2) must be true. Furthermore, if at least one of $\mapA^{-1}(\K)$ or $C$ is not a subspace, then exactly one of the two statements is true.
%$X(\zeros)=\{\zeros\}$ implies that $\strict{Y(\zeros)} \neq \emptyset$. Furthermore, the converse is also true when $X(\zeros)$ is not a subspace.
\end{lemma}

Different versions of this theorem of the alternative have appeared for e.g. in \citep[Corollary~2]{luo1997duality} and \citep[Lemma~2.1]{schurr2007universal}, but our version is more general as it allows for two cones $\K$ and $C$ and also notes the non-subspace requirement to obtain equivalence. Note that the non-subspace condition is needed to obtain the reverse direction $(2) \implies \neg ~ (1)$, because if this assumption does not hold, then the simple example from linear programming where the two systems are $\{(x_{1},x_{2})\in\real^{2}\colon x_{1}\le 0, \ -x_{1}\le 0 \}$ and $\{(y_{1},y_{2})\in\real^{2}_{+}\colon y_{1}=y_{2}\}$, tells us that both statements can be simultaneously true.

The proofs for the previous two lemmas are given in Appendix~\ref{sec:missing} for completeness.

\section{Basic Results}	\label{sec:basic}

Assuming feasibility, note that $\cstar = +\infty$ if  $c\notin(\lineal{X})^{\bot}$ and $\optdual = -\infty$ if $b\notin(\lineal{Y})^{\bot}$. These conditions are trivially satisfied when the respective feasible sets are pointed, but they are not necessary for unboundedness. Because $(\lineal{X})^{\bot} = \dualcone{(\lineal{X})}\cap\polar{(\lineal{X})}\subset\polar{(\rec{X})}$ and similarly $(\lineal{Y})^{\bot}\subset\polar{(\rec{Y})}$, another sufficient condition for unboundedness that is also not necessary in general is $c\notin\polar{(\rec{X})}$ and $b\notin\polar{(\rec{Y})}$, respectively. Trivial cases for computing $\cstar$ and $\optdual$ would be when the objective functions are constant-valued over the respective feasible sets. For the primal, $\inner{c}{x}$ is constant-valued over $X$ if and only if there exists some $u\in\E$ such that $u+c^{\bot}\supset X$, and for the dual, $\inner{b}{y}$ is constant-valued over $Y$ if and only if there exists some $u\in\EE$ such that $u+b^{\bot}\supset Y$. 
%If the objective not constant-valued over a feasible set, then the optimum can be {\tajayi{approximated arbitrarily well}}, i.e., for small $\epsilon > 0$ {\tajayi{there exists an $x \in X$ with $\inner{c}{x} \geq \cstar - \epsilon$}} ($\epsilon$-suboptimal solution), and similarly for the dual. 

Now let us give some detailed observations on conic optimization.

\subsection{Basis Representation of the Dual}

%Let us first make an observation about the formulation of the dual problem. 
Ordinarily, as seen in~\eqref{def:dual}, we use the adjoint of $\mapA$ to write the constraints of the dual. When given a basis for the span of $C$, the dual problem can be written using a linear map that depends on the basis, instead of using the adjoint of $\mapA$.

\begin{proposition}	\thlabel{dualbasis}
Let $B=\{v^{1},\dots,v^{m}\}$ be an orthonormal basis of $\linhull{C}$ and denote the linear map $\mapA[B]\colon y \in\EE \mapsto \sum_{j=1}^{m}\inner{\A[v^{j}]}{y}v^{j} \in \linhull{C}$. We have $\optdual = \inf\{\inner{b}{y}\colon \mapA[B]y \conegeq[\dualcone{C}] c,\ y\in\dualcone{\K}\}$.
\end{proposition}
\begin{proof}
Any $x\in C$ can be written as $x = \sum_{i=1}^{m}\alpha_{i}v^{i}$ for some $\alpha\in\real^{m}$. We have
\begin{align*}
\inner{x}{\mapA[B]y}[\E] \eq \inner{\sum_{i=1}^{m}\alpha_{i}v^{i}}{\sum_{j=1}^{m}\inner{\A[v^{j}]}{y}[\EE]v^{j}}[\E] &\eq \sum_{i=1}^{m}\sum_{j=1}^{m}\alpha_{i}\inner{\A[v^{j}]}{y}[\EE]\inner{v^{i}}{v^{j}}[\E] \\
&\eq \sum_{j=1}^{m}\alpha_{j}\inner{\A[v^{j}]}{y}[\EE] \eq \inner{\A[(\sum_{j=1}^{m}\alpha_{j}v^{j})]}{y}[\EE],
\end{align*}
where the penultimate equality is because of orthonormality of the basis. Thus,
\begin{equation} \label{dualbasiseq}
\inner{\A[x]}{y}[\EE] \eq \inner{x}{\mapA[B]y}[\E], \quad x\in C, y\in \EE.
\end{equation}

Since $\optdual$ is the Lagrangian dual of the primal, we have $\optdual = \inf_{y\in\dualcone{\K}}\sup_{x\in C}\inner{c}{x}[\E] + \inner{y}{b - \A[x]}[\EE]$. The inner maximization objective becomes $\inner{b}{y}[\EE] + \inner{c}{x}[\E] - \inner{y}{\A[x]}[\EE]$, and then equation~\eqref{dualbasiseq} transforms this to  $\inner{b}{y}[\EE] + \inner{c - \mapA[B]y}{x}[\E]$. Hence, the inner maximum is finite (and equal to zero) if and only if $c - \mapA[B]y \in \polar{C}$, which is the dual constraint $\mapA[B]y \conegeq[\dualcone{C}]c$.
\end{proof}

The implication of this is that any of the results in this paper that make use of the dual constraints can be restated in terms of the basis form of the dual constraints using a basis for the span of $C$.

\subsection{Recession Cone}

Our second basic result is characterising the recession cone of the feasible regions. This will also be helpful later in the paper when deriving conditions for boundedness since a nonempty closed convex set is bounded if and only if its recession cone contains nonzero elements.

\begin{lemma}	\thlabel{recconeX} % \thlabel{polarstrict}
For $X\neq\emptyset$, we have $\rec{X} = {X(\zeros)} = -\mapA^{-1}(\K)  \cap {C}$ and
\[
\polar{(X(\zeros))} \eq \closure{\dualfeas}, \qquad \text{and} \qquad  \relint{\polar{(X(\zeros))}} \eq \relint{\dualfeas}.
\]
%\begin{multline*}
%\closure{\dualfeas} \eq \polar{(\rec{X})} \,\supseteq\, \{c\in\E\colon Y(c) \text{ is almost feasible}\} \,\supseteq\, \relint{\polar{(\rec{X})}} \eq \relint{\dualfeas} \\ \eq \{c\in\E\colon \strict{Y(c)}\neq\emptyset\}.
%\end{multline*}
Furthermore, $\polar{(X(\zeros))} = \dualfeas$ when there exists $x\in\relint{C}$ with $-\A[x]\in\relint{\K}$. %$\strict{(\rec{X})}\neq\emptyset$.
\end{lemma}
\begin{proof}
We will use the fact that for an affine map $\G$, the $\rec{}$ operator %does not commute with the preimage $\G^{-1}$ but instead it 
commutes with the linear map associated with $\G$, i.e., $\rec{\G^{-1}(\X)} = -\L_{\G}^{-1}(\rec{\X})$, where $\G$ is an affine map  with $\G^{-1}(\X)\neq\emptyset$ for a closed convex set $\X$, and $\L_{\G}\colon x\mapsto\G(\zeros)-\G(x)$. 
%\begin{claim}	\thlabel{recaffinv}
%$\rec{\G^{-1}(\X)} = -\L_{\G}^{-1}(\rec{\X})$, where $\G$ is an affine map  with $\G^{-1}(\X)\neq\emptyset$ for a closed convex set $\X$, and $\L_{\G}\colon x\mapsto\G(\zeros)-\G(x)$.
%\end{claim}
%\claimproof{
%Because $\G(x)=\G(\zeros)-\L_{\G}(x)$, we have $\G^{-1}(\X) = \L_{\G}^{-1}(\G(\zeros)-\X)$. Commutativity of $\rec{}$ with $\L_{\G}^{-1}$ gives us $\rec{\G^{-1}(\X)} = \L_{\G}^{-1}(\rec{(\G(\zeros)-\X)})$. Translation invariance of $\rec{}$ means that $\rec{(\G(\zeros)-\X)} = \rec{(-\X)} = -\rec{\X}$, and so we have our claim.
%}
Using the affine preimage expression for $X$ from \eqref{feaspreimg}, the above claim yields $\rec{X} = -\L_{G_{p}}^{-1}(\rec{(\K\times C)})$. Distributivity of $\rec{}$ over Cartesian product and $\K$ and $C$ being cones leads to $\rec{X} = - \L_{G_{p}}^{-1}(\K\times C)$, which means that $\rec{X} \eq \{x \in C \sep \A[x] \coneleq \zeros \}$. The set on the right-hand side is equal to $X(\zeros)$ from \eqref{def:XY} and can be rewritten as $-\mapA^{-1}(\K)  \cap {C}$.
%indeed equal to the recession cone of $X(\zeros)$ since the above claim makes our arguments invariant to the right-hand side in the conic constraints.

%So let us argue that the polar cone is equal to the closure of the dual feasibility cone. 
Since $X(\zeros) = \L_{G_{p}}^{-1}(-\K\times- {C})$, \mythref{polarinv} implies that the polar of $X(\zeros)$ is equal to the closure of $\adjoint{\L}_{G_{p}}(\dualcone{\K}\times\dualcone{C})$, and from \eqref{feascones} we know that this is the closure of $\dualfeas$. The equality for relative interiors follows from the assertion on closure because of the first equality in \mythref{facts} for arbitrary convex sets. 
%Because the adjoint of $\L_{G}$ is the linear map $(y,w)\mapsto\mapAad[y]-w$, we obtain that $\polar{(\rec{X})} = \mapAad(\dualcone{\K}) - \dualcone{C}$, and then the definition of $\dualfeas$ in \eqref{feascones} gives us the equality of the polar cone to $\closure{\dualfeas}$. 
%Applying the Bipolar Theorem to this equality and using the fact that the polars of $\dualfeas$ and $\closure{\dualfeas}$ are equal leads us to $\rec{X}$ being equal to the polar of $\dualfeas$.

For the last part, it suffices to argue that $\dualfeas$ is a closed set. The existence of a strict solution in the recession cone means that $\mapA^{-1}(\relint{\K})\neq\emptyset$, and then known conditions for closeness of an adjoint image (cf.~\mythref{linimg}) give us that $\mapAad(\dualcone{\K})$ is a closed cone. Therefore, $\dualfeas$ is a Minkowski difference of two closed cones. \mythref{polarinv} tells us that $\dualfeas$ is the Minkowski sum of the polars of $\mapA^{-1}(-\K)$ and $C$. For this sum to be closed, a sufficient condition from literature (cf.~\citep{pataki2007closedness}) is that the relative interiors of $\A^{-1}(-\K)$ and $C$ intersect. We have $\relint{\A^{-1}(-\K)}\supseteq\A^{-1}(-\relint{\K})$. Strict feasibility implies $\A^{-1}(-\relint{\K})\cap\relint{C}\neq\emptyset$, and therefore, $\dualfeas$ is a closed set.
\end{proof}

%The descriptions of $\rec{X}$ and $\polar{(\rec{X})}$ from \textsection\ref{sec:reccone} lead to primal and dual characterizations for boundedness of $X$. The closed convex set $X$ is bounded if and only if $\rec{X}=\{\zeros\}$, and so it is obvious from \hyperref[recconeX]{\thref{recconeX}} that a characterization of unboundedness of $X$ is that there exists $\zeros\neq x\in C$ such that $\A[x]\coneleq \zeros$. 

\begin{corollary}
A nonempty $X$ is bounded if and only if for any norm $\|\cdot\|$ the convex maximization problem $\max\{\|x\| \colon \A[x] \coneleq \zeros, \ x\in C \}$ has value greater than $0$.
\end{corollary}

Henceforth, whenever we refer to a feasible set being bounded, it is assumed that the set is nonempty.

\subsection{Objective Qualification}

The third basic result to note here is a special case of strong duality that does not require any constraint qualification, instead it requires the objective vectors to belong to specific parts of feasibility cones of the other problem.

\begin{observation}	\thlabel{trivialc}
If the primal is feasible, then strong duality holds when $c\in\mapAad(\K^{\bot})$.
\end{observation}
\begin{proof}
Let $c=\adjoint{\A}(y)$ for some $y\in\K^{\bot}$. Because $\adjoint{\A}(\K^{\bot})\subseteq\mapAad(\dualcone{\K})\subseteq\dualfeas$, we have $c\in\dualfeas$ and hence the dual is feasible with $y\in Y$. Every $x\in X$ has $\A[x] + s = b$ for some $s\in\K$. Then, $\inner{c}{x} = \inner{\mapAad(y)}{x} = \inner{y}{\A[x]} = \inner{y}{b-s} = \inner{y}{b} - \inner{y}{s} = \inner{y}{b}$, where the last equality is due to $y\in\K^{\bot}$ and $s\in\K$. Thus, there is zero duality gap and solvability of the dual.
\end{proof}

After using \mythref{spandual}, the analogous condition for the dual becomes $b\in\A(\lineal{C})$.

\section{Closedness CQ and its Consequences}	\label{sec:closed}

Consider the linear maps
\begin{equation}	\label{def:LM}
\Lp\colon (\alpha,\alpha_{0})\in\E\times\real \mapsto (\A[\alpha]+\alpha_{0}b, \, -\alpha)  \in \EE\times\E, \qquad \adjLp \colon (y,w)\mapsto (\mapAad[y] - w, \, \inner{b}{y}).
\end{equation}
The convex cone $\adjLp(\dualcone{\K}\times\dualcone{C})$ may not be a closed set in general because linear images of closed sets are not closed (cf.~\mythref{linimg}). Closedness of this adjoint image is known to be a sufficient condition for strong duality. 

\begin{subequations}	\label{closedness}
\begin{theorem}[{\citet[Proposition 2.6]{Shapiro_2001}, \citet[Theorem 7.2]{barvinok2002course}}]	\thlabel{strong}
Assume the primal is feasible. Strong duality holds, with the dual being solvable, when 
\begin{equation}	\label{primclosed}
\adjLp(\dualcone{\K}\times\dualcone{C}) \ \text{ is a closed set.} 
\end{equation}
\end{theorem}

\begin{remark}	\thlabel{dualclosedrem}
When the dual is feasible, primal-dual symmetry implies that an analogous sufficient condition for strong duality is 
\begin{equation}	\label{dualclosed}
\adjLd(C\times\K) \ \text{ being a closed set, where } \ 
\adjLd \colon (x,s)\mapsto (\A[x]+s,\, \inner{c}{x})
\end{equation}
is the adjoint of the linear map $\Ld\colon (\beta,\beta_{0})\in\EE\times\real \mapsto (\mapAad[\beta]+\beta_{0}c, \, \beta) \in \E\times\EE$.  We refer to closedness of $\adjLp(\dualcone{\K}\times\dualcone{C})$ as the dual condition because the adjoint image relates to dual feasibility and optimality, 
%because it takes the adjoint of the map $\Lp$ which corresponds to primal feasibility. Similarly, 
and closedness of  $\adjLd(C\times\K)$ as the primal condition.
\end{remark}
\end{subequations}

\begin{remark}	\thlabel{closednessrem}
The Closedness CQs in \eqref{closedness} are different than another closedness CQ given by \citet[Theorem 3.2]{boct2006alternative}; the latter was shown to be \emph{necessary and} sufficient for strong duality in abstract convex programs. It is not clear whether \eqref{closedness} is also necessary for strong duality; we leave its relationship to the CQ of \citeauthor{boct2006alternative} as an open question.
\end{remark}

The proofs of \mythref{strong} by \citeauthor{Shapiro_2001} and by \citeauthor{barvinok2002course} are different in their approaches --- the former uses functional analysis arguments for the value function whereas the latter uses geometric arguments using the separation theorem for closed convex sets. For the sake of exposition, we give a sketch of the latter proof by breaking down its main components.

\begin{remark}[On \citeauthor{barvinok2002course}'s proof of \mythref{strong}]	\thlabel{barvinok}
The proof can be essentially broken down into three main steps. For convenience, denote $\X := \adjLp(\dualcone{\K}\times\dualcone{C})$.
\begin{enumerate}
\item First is to argue that when $\optdual$ is finite, which happens  when primal is assumed to be feasible, then $(c,\optdual)$ belongs to the closure of $\X$. This is shown by observing that $\optdual$ is the infimum of values over a closed interval of the real line; in particular, $\optdual = \inf_{t}\{t\colon t\in I\}$ where $I = \{t\in\real\colon (c,t)\in\closure\X\}$.

\item Note that although $(c,\optdual)$ belongs to $\closure\X$. there is no guarantee that this point belongs to $\X$, and this inclusion is equivalent to solvability of the dual due to the definition of the adjoint image in \eqref{def:LM}. In particular, note that if $(c,z) \in \X$, then it must be that $c\in\dualfeas$ and $z \ge \optdual$; this is because of the definition of $\dualfeas$ in \eqref{feascones} and $\optdual = \inf\{\inner{b}{y}\colon c\in\dualfeas \}$. 
% (existence of $y\in\dualcone{\K}$ such that $\mapAad[y] - c \in \dualcone{C}$ and $\inner{b}{y}=\optdual$)
Assuming closedness of the adjoint image gives us $(c,\optdual)\in \X$, and since this is dual solvability, it follows that $(c,\optdual-\epsilon) \notin \X$ for every $\epsilon > 0$, because otherwise we would have a contradiction to the optimality of the solution.

\item The third, and main, step uses separation theorem to show that the condition $(c,\optdual-\epsilon) \notin \closure\X$, for some $\epsilon > 0$, is sufficient to guarantee an $\epsilon$-gap between the primal and dual, i.e., $\optdual - \cstar\le\epsilon$. Using the closedness assumption and the second claim that $(c,\optdual-\epsilon) \notin \X$ for every $\epsilon > 0$, leads to $\optdual - \cstar\le \inf\{\epsilon\colon \epsilon > 0\} = 0$, and then $\optdual - \cstar \ge 0$ from weak duality implies that $\cstar = \optdual$, thereby establishing strong duality. \hfill $\diamond$
\end{enumerate}
\end{remark}

The third claim in the above remark is a topological necessary condition for one problem to have finite value and the other problem to be infeasible (resulting in infinite duality gap). 

\begin{corollary}	
Suppose that the dual is solvable and the primal is infeasible. Then it must be that for every $\epsilon > 0$ we have $(c,\optdual-\epsilon)\in\partial\X$, where $\partial\X := \closure\X\setminus\X$ is the boundary of $\X:=\adjLp(\dualcone{\K}\times\dualcone{C})$. 
\end{corollary}
%\begin{proof}
%Follows from the contrapositive statement to the third claim in \mythref{barvinok} and the fact that dual solvability is equivalent to $(c,\optdual)\in\X$.
%%reads that $(c,\optdual-\epsilon)$ belongs to the closure of $\X$ for every $\epsilon > 0$.
%\end{proof}

Let us illustrate this necessary condition with a family of examples that have infinite duality gap in every dimension.

\begin{example} \thlabel{exampleAdapted}
For arbitrary $n\ge 3$, let $A = \left[\begin{array}{cc}I_{n-1} & e_{1}  \end{array}\right]$ be an $(n-1)\times n$ matrix, where $I_{n-1}$ is an identity matrix and $e_{1}$ is the column vector $(1,0,\ldots,0)^{\top}\in\real^{n-1}$, and $\C_{n} = \{x\in\real^{n}\colon x_{n}\ge \|(x_{1},\dots,x_{n-1}) \|_{2} \}$ be the Lorentz cone in $\real^{n}$. Consider the primal-dual pair 
\[
\cstar \eq \sup\left\{0 \colon Ax = \ones - e_{1}, \ x\in\C_{n} \right\}, \quad  \optdual \eq \inf\left\{y_{2} + \cdots+ y_{n-1}\colon A^{\top}y \conegeq[\C_{n}] \zeros, \ y\in\real^{n-1} \right\},
\]
With respect to the primal formulation in \eqref{def:primal} we have $\K=\{\zeros\}$, $C=\C_{n}$, $b = \ones - e_{1}$ and $c=\zeros$. The primal constraints $x_{j}=1, j=2,\dots,n-1$ imply that for $x\in\C_{n}$, we must have $x_{n}\ge \sqrt{x_{1}^{2} + n-2}$, but then the first primal constraint $x_{1}+x_{n}=0$ makes the problem infeasible. Hence, $\cstar = -\infty$. The dual is obviously feasible with $y=\zeros$. In fact, the dual optimum is $\optdual = 0$ because $A^{\top}y = (y_{1},\dots,y_{n-1},y_{1})^{\top}$, and so every dual feasible solution has $y_{1}\ge \|(y_{1},\dots,y_{n-1}) \|_{2}$, which implies that $y_{j}=0, j=2,\dots,n-1$. 

The primal linear map and its adjoint from \eqref{def:LM} are $\Lp(\alpha,\alpha_{0}) = (A\alpha + \alpha_{0}(\ones - e_{1}), -\alpha)$ and $\adjLp(y,w) = (A^{\top}y - w, y_{2} + \cdots+ y_{n-1})$. %We claim that for any $\epsilon > 0$ the point $(\zeros,-\epsilon)$, which is equal to $(c,\optdual-\epsilon)$, does not belong to $\closure{\adjLp(\real^{n-1}\times\C_{3})}$.
Denote $\X = \closure{\adjLp(\dualcone{\{\zeros\}}\times\dualcone{\C}_{n})}$ and note that this is a closed convex cone. Because the Lorentz cone $\C_{n}$ is self-dual, we get that $\X$ is equal to $\closure{\left\{(A^{\top}y - w, y_{2} + \cdots+ y_{n-1})\colon y \in \real^{n-1}, w \in \C_{n} \right\}}$. Take any $\epsilon > 0$. By the separation theorem, the point $(\zeros,-\epsilon)$, which is equal to $(c,\optdual-\epsilon)$, does not belong to $\X$ if and only if there exists $(\alpha,\alpha_{0})\in\polar{\X}$ such that $\inner{\zeros}{\alpha} - \epsilon\alpha_{0} > 0$. Suppose there exists $(\alpha,\alpha_{0})\in\polar{\X}$ with $\epsilon\alpha_{0} < 0$. \mythref{polarinv} tells us that $\polar{\X} = \Lp^{-1}(-\{\zeros\}\times-\C_{n})$. Therefore, it must be that $\Lp(\alpha,\alpha_{0})\in-\{\zeros\}\times-\C_{n}$, which implies that $A\alpha + \alpha_{0}(\ones - e_{1}) = 0$ and $\alpha\in\C_{n}$. The linear equation can be scaled by $-\alpha_{0}$ to get $A\alpha^{\prime} = \ones - e_{1}$ for $\alpha^{\prime} = -\alpha/\alpha_{0}$. Because $\alpha_{0}<0$ due to $\epsilon > 0$ and $\epsilon\alpha_{0}<0$, we have $\alpha^{\prime}\in\C_{n}$ whenever $\alpha\in\C_{n}$. Hence, there must exist some $\alpha^{\prime}\in\C_{n}$ for which $A\alpha^{\prime}=\ones - e_{1}$. However, this is exactly primal feasibility and we already argued that the primal is infeasible, therefore giving us a contradiction to the existence of $\alpha^{\prime}$. Hence, $(\zeros,-\epsilon)$ cannot be separated from the closure of $\adjLp(\real^{n-1}\times\C_{n})$ for any $\epsilon > 0$.
\end{example}

%\begin{remark}
%Why do we need closedness of either $\K$ or $C$ ??? \akshayq{fill in}
%\end{remark}

Now we show two sufficient conditions for strong duality that emerge as special cases of the closedness condition. Both of these conditions are known in literature through different proofs. 
%The basic ingredient of our derivation is known sufficient conditions for the closedness of the linear image of a cone. 
The first such condition is the well-known constraint qualification related to strict feasibility of one of the problems, whereas the second condition is boundedness of the feasible region and is perhaps less well-known.

%\mythref{strong} and show how emerge from the general closedness condition. Some of these conditions are related to imposing a constraint qualification in the problem. 

%These four conditions on closedness imply four sufficient conditions for strong duality. First, we have the well-known condition of strict feasibility (cf.~\eqref{strictfeas}), also referred to as generalized Slater CQ, which was directly proven by \citep[Theorem 1.7.1]{bentalnem2001} in the full-dimensional case, but now we see it as an immediate consequence of sufficiency of the closedness condition in the general case. 

\subsection{Slater CQ}

A problem has the \emph{generalized Slater's constraint qualification} (Slater CQ) when it has strictly feasible solutions, where the sets of such solutions to the primal and dual problems are
\begin{subequations} \label{strictfeas}
\begin{align}	
\strict{X} &\eq \strict{X(b)} \define \left\{x \in \relint{C} \sep \A[x] \conel b \right\},	  \\
\strict{Y} &\eq \strict{Y(c)} \define \left\{y\in\relint{\dualcone{\K}} \sep \mapAad[y] \coneg[\dualcone{C}] c \right\}. 
\end{align}
%Strictly feasible solutions are not only important for guaranteeing strong duality but also to describe the relative interior of the polar of the recession cone. 
\end{subequations}
Thus, the primal (resp. dual) has Slater CQ if and only if $\strict{X}\neq\emptyset$ (resp. $\strict{Y}\neq\emptyset$). From \mythref{genfeas} we know that a problem has Slater CQ if and only if the right-hand side of the conic inequality constraints belongs to the relative interior of the feasibility cone for that problem. 

A related notion to strictly feasible points is that of points in the relative interior of $X$, so the question arises whether the two sets $\relint{X}$ and $\strict{X}$ are equal. We know from basic convex analysis that the relative interior of a set is a topological concept and is independent of algebraic representation of the set, and nonempty convex sets have nonempty relative interiors. However, strictly feasible solutions may exist for one algebraic formulation of the set but not for the other. For example, $X = \{x\in\real^{n}_{+}\colon a^{\top}x \le 1,\, -a^{\top}x \le -1  \}$, for some vector $a > \zeros$, has $\strict{X}=\emptyset$, but writing the same set as $X = \{x\in\real^{n}_{+}\colon a^{\top}x = 1\}$ gives $\strict{X} = \{x>\zeros\colon a^{\top}x = 1\}$. When strictly feasible solutions do exist, they indeed form the relative interior of the set.

\begin{observation}	\thlabel{strictfeasrelint}
For $X\neq\emptyset$, we have $\relint{X} = \strict{X}$ if and only if $\strict{X}\neq\emptyset$.
\end{observation}

The arguments for this are elementary and presented in Appendix~\ref{sec:missing} for sake of completeness and because we did not find an explicit reference in literature.

%\begin{equation}
%(\text{Primal Slater CQ for $X$}) \sep \qquad \strict{X} \neq \emptyset,
%\end{equation}
%and one can define this CQ for the dual also. 

The significance of Slater CQ is that it is well-known to be a sufficient condition for achieving strong duality in convex optimization; there are many proofs of this in literature, for e.g. \citep[Theorem~28.2]{rockafellar1970convex} and \citep[Theorem~11.15 and Remark~11.16]{guler2010found}. For conic problems, there are different specialised proofs of strong duality under Slater CQ, such as \citep[Theorem~7]{luo1997duality}, \citep[Theorem~1.7.1]{bentalnem2001}, \citep[Theorem~4.14]{rusz2006nonlinear}, and \citep[Corollary~4.8]{tunccel2012strong}, and it can also be derived directly from the powerful Fenchel Duality Theorem in convex analysis \citep[Theorem~31.4]{rockafellar1970convex}. We state the result explicitly since it will be used later in this paper.

\begin{theorem}[Slater CQ]		\thlabel{slaterstrong}
If either $\strict{X}\neq\emptyset$ or $\strict{Y}\neq\emptyset$, then strong duality holds. Furthermore, the dual (resp. primal) has an optimal solution when the primal (resp. dual) has Slater CQ.
\end{theorem}

We now present another proof for sufficiency of Slater CQ for strong conic duality by showing that this CQ is a stronger condition than the closedness criteria stated in equations~\eqref{closedness}. The key ingredient of this argument is sufficient conditions for a linear image of a closed convex cone to be a closed set.

\begin{lemma}[cf.~{\citep{pataki2007closedness}}]	\thlabel{linimg}
Given a linear map $\L$ and a non-polyhedral closed convex cone $\C$, the linear image $\adjoint{\L}(\dualcone{\C})$ is a closed set if any of the following conditions hold: %$\C$ is a polyhedral cone, %or $\range{\L}\cap\relint{\C}\neq\emptyset$, or 
\begin{enumerate}
\item $\range{\L}\cap \relint{\C} \neq \emptyset$, %$\L^{-1}(\relint{\C})\neq\emptyset$,
\item $\kernel\adjoint{\L}\cap\relint{\dualcone{\C}}\neq\emptyset$,
\item $\range{\L}\cap\C\subseteq\lineal{\C}$,
\item $\kernel\adjoint{\L}\cap\dualcone{\C}\subseteq\C^{\bot}$. \footnote{Since $\C^{\bot}=\polar{\C}\cap\dualcone{\C}$, a stronger version of the last condition would require that $\kernel\adjoint{\L}\subseteq\polar{\C}$.}  
\end{enumerate}
\end{lemma}

\begin{proposition}
%If the primal has Slater CQ, then $\adjLp(\dualcone{\K}\times\dualcone{C})$ is a closed set. 
If a problem has Slater CQ then the corresponding closedness condition (either~\eqref{primclosed} or \eqref{dualclosed}) is satisfied and consequently, strong duality holds.
\end{proposition}
\begin{proof}
If closedness condition is satisfied (either for primal or dual), then \mythref{strong} and \mythref{dualclosedrem} tell us that strong duality holds. So we have to argue that if $\strict{X}\neq\emptyset$ (primal Slater CQ) then $\adjLp(\dualcone{\K}\times\dualcone{C})$ is a closed set. We will need the following result which says that the perspective image of the preimage of a cone under the linear map $\Lp$ is exactly the preimage of a cone under the affine map $G\colon x \mapsto (b -\A[x],x)$.

\begin{claim}	\thlabel{perspimg}
Let $\C$ and $\CC$ be two nonempty convex cones, and $\X_{1} = \{(\alpha,\alpha_{0})\in \Lp^{-1}(\CC\times\C)\sep \alpha_{0} > 0 \}$ and $\X_{2} =  G^{-1}(\CC\times\C)$. Denote the perspective map by $\persp:(\alpha,\alpha_{0})\in\E\times\real\setminus\{0\} \mapsto \alpha/\alpha_{0} \in \E$.

Then, $-\persp(\X_{1}) = \X_{2}$.
\end{claim}
\claimproof{
Take any $(\alpha,\alpha_{0})\in\X_{1}$. We have $-\alpha/\alpha_{0}\in\C$ due to $-\alpha\in\C$ and $\alpha_{0}>0$. Also, $\alpha_{0}(b - (\A[(-\alpha/\alpha_{0})])) = \A[\alpha]+\alpha_{0}b\in\CC$ and $\alpha_{0}>0$ implies that $b - (\A[(-\alpha/\alpha_{0})]) \in \CC$. Hence, $-\persp(\X_{1}) \subseteq \X_{2}$. The reverse inclusion $-\persp(\X_{1}) \supseteq \X_{2}$ follows from $G(x) = \Lp(-x,1)$.
}

Because $X = G_{p}^{-1}(\K\times C)$ as per \eqref{feaspreimg}, this claim implies that $X$ is the negative perspective image of the preimage of $\K\times C$ under the map $\Lp$. Because $\strict{X} = G^{-1}(\relint{\K}\times\relint{C})$, applying \mythref{perspimg} with $\CC=\relint{\K}$ and $\C=\relint{C}$ gives us that  $\strict{X}\neq\emptyset$ implies $\Lp^{-1}(\relint{\K}\times\relint{C})\neq\emptyset$. Polarity and $\relint{}$ operators distribute over the Cartesian product, and so $\dualcone{(\K\times C)} = \dualcone{\K}\times\dualcone{C}$ and $\relint{(\K\times C)}=\relint{\K}\times\relint{C}$. Using the first condition in \mythref{linimg} with $\C=\K\times C$ implies that $\adjLp(\dualcone{\K}\times\dualcone{C})$ is a closed set.
\end{proof}

\subsection{Existence of Slater CQ}

We present some necessary and sufficient conditions for Slater CQ to hold.  The definition of strict feasibility tells us that strictly feasible solutions exist only when $X\cap\relint{C}\neq\emptyset$ and so $X\not\subset\boundr{C}$ is necessary for their existence. But this is far from being a sufficient condition. The example $X = \{x\in\real^{n}_{+}\colon a^{\top}x \le 1,\, -a^{\top}x \le -1  \}$ mentioned earlier satisfies $X\not\subset\cup_{j=1}^{n}\{x\colon x_{j}=0\} = \boundr{\real^{n}_{+}}$, but it does not have any strictly feasible solutions because the slacks of the two inequality constraints are negative of each other, implying that $X$ is contained in the subspace formed by the slack variables equal to zero.  We show that the slacks of conic constraints forming a full-dimensional set is a sufficient condition for strict feasibility, and  a necessary condition, under a technicality, is that $X$ have the same dimension as $C$. Define the set of slack values for $X$ as \[\slack{X} \define b - \A(X) \eq \left\{b -\A[x]\sep x\in X\right\}\] and let $\dimn{\cdot}$ denote the affine dimension of a set.

\begin{proposition}		\thlabel{strictfeasrelint2}
We have the following.
\begin{enumerate}
\item $\strict{X}\neq\emptyset$ if $\aff{(\slack{X})} = \linhull{\K}$ and $\emptyset\neq X\not\subset\boundr{C}$.
\item When $\A(\linhull{C})\subseteq\linhull{\K}$, $\strict{X}\neq\emptyset$ only if $\dimn{X} = \dimn{C}$. 
\end{enumerate}
\end{proposition}
\begin{proof}
%We argue the two claims by contraposition. 

%The necessary and sufficient conditions on non-existence of strict feasible solutions can be argued by contraposition. First we argue necessary condition when $\XX$ is full-dimensional in $\real^{n}$. 
(1) For a nonempty convex set $X$, we have that $\slack{X}$, which is the affine image $b-\A(X)$, is also a nonempty convex set, and therefore $\relint{(\slack{X})} \neq\emptyset$. Because $\slack{X}\subset\K$, the assumption of equal affine hulls for $\K$ and $\slack{X}$ implies that $\relint{(\slack{X})}\subseteq\relint{\K}$. The commutativity of $\relint{}$ and affine images of convex sets gives us $\relint{(\slack{X})} = b - \A(\relint{X})$. Hence, $b - \A(\relint{X})\subseteq\relint{\K}$. When $X\not\subset\boundr{C}$, we have $\relint{X}\cap\relint{C}\neq\emptyset$ because otherwise $X\subset C = \relint{C}\cup\boundr{C}$, $\closure{(\relint{X})} = X$, and $\boundr{C}$ being a closed set implies the contradiction $X\subset\boundr{C}$. Therefore, there exists some $x\in\relint{C}$ with $b-\A[x]\in\relint{\K}$, and hence $\strict{X}\neq\emptyset$. 

(2) Clearly, $\dimn{X} \le \dimn{C}$ because $X\subseteq C$. Suppose $\strict{X}\neq\emptyset$. Take any $x\in\strict{X}$ and $y\in\linhull{C}$. We have $\A[y] \in \linhull{\K}$ due to the assumption that $\A(\linhull{C})\subseteq\linhull{\K}$. Because $\strict{X} = \A^{-1}(b-\relint{\K})\cap \relint{C}$, there exists some $\epsilon > 0$ such that $x+\epsilon y \in C$ and $b-\A[x] + \epsilon \A[y]\in\K$. The linearity of $\A$ makes $b-\A[x] + \epsilon \A[y]\in\K$ equivalent to $b - \A(x + \epsilon y) \in \K$, and so $x+\epsilon y \in X$. Therefore, for any basis $\{y^{1},\dots,y^{k}\}$ of the subspace $\linhull{C}$, the points $\{x, x + \epsilon y^{1},\dots, x + \epsilon y^{k}\}$ are affinely independent in $X$, implying that $\dimn{X} \ge \dimn{C}$, as desired.
\end{proof}

This leads to a case where full-dimensionality of $X$ is necessary and sufficient for the existence of strictly feasible solutions. 
%We will need the basic fact that .
%
%\begin{lemma}	\thlabel{lineardim}
%A linear map $\L$ and set $\X$ have $\dimn{\L(\X)} \le \dimn{\X}$, with equality holding when $\L$ is injective. 
%\end{lemma}

\begin{proposition}
Suppose $C$ is full-dimensional and $\A$ is an injective map with $\range{\A}\subseteq\linhull{\K}$. Then $\strict{X}\neq\emptyset$ if and only if $X\not\subset\boundr{C}$ and $X$ is a full-dimensional set. 
\end{proposition}
\begin{proof}
The only if direction is directly from the second claim in  \hyperref[strictfeasrelint2]{\thref{strictfeasrelint2}}, whereas the if direction is from the first claim in the proposition. To see the if direction, suppose $X$ is a full-dimensional set. We have $\aff{(\slack{X})} = \aff{(b-\A(X))} = b - \aff{\A(X)} =  b - \A(\aff{X})$, which implies $\dimn{(\slack{X})} = \dimn{\A(\aff{X})}$. Recall the basic fact from linear algebra that taking a linear image of a set is a dimension-reducing operation in general, but the dimension is preserved when the linear map is injective. Therefore, when $\A$ is injective, we have $\dimn{\A(\aff{X})} = \dimn{(\aff{X})} = \dimn{X}$. Now, $\dimn{(\slack{X})} = \dimn{\A(\aff{X})}$ argued earlier and full-dimensionality of $X$ implies that $\slack{X}$ is also full-dimensional. Since $\slack{X}\subset\K$, we have $\aff{(\slack{X})}\subseteq\linhull{\K}$ and therefore, full-dimensionality of $\slack{X}$ leads to $\aff{(\slack{X})} = \linhull{\K} = \EE$.
\end{proof}

Our last sufficient condition for existence of Slater CQ is in terms of the homogenous set (cf.~\eqref{def:XY}) $X(\zeros) := \{x\in C\colon\A[x]\coneleq\zeros\}$, which we know from  \mythref{recconeX} to be the recession cone whenever the right-hand side gives feasibility. We argue that when this set is strictly feasible, which by applying equation~\eqref{strictfeas} with $b = \zeros$ means that the strictly feasible solutions in the homogenous set can be described as
\begin{equation}	\label{strictrecX}
\strict{X(\zeros)} \eq -\mapA^{-1}(\relint{\K})\cap\relint{C} \eq \left\{x\in\relint{C}\sep \A[x]\conel \zeros \right\},
\end{equation}
the convex cone $\primalfeas$ of all feasible right-hand sides for the primal is open and contains the span of $\K$. This implies that whenever the primal is feasible, it is also strictly feasible, and that the primal is feasible for all right-hand sides in the span of $\K$.

\begin{proposition}		\thlabel{strictrecXnot}
Suppose that $\strict{X(\zeros)}\neq\emptyset$. Then $\primalfeas = \relint{\primalfeas} \supseteq \linhull{\K}$; in particular, $\strict{X}\neq\emptyset$ if and only if $X$ is feasible, and $X$ is strictly feasible for $b\in\linhull{\K}$.
\end{proposition}
\begin{proof}
For the equality we have to argue the $\subseteq$-inclusion. Take $y\in\strict{X(\zeros)}$ and $x\in X(b)$ for some $b\in\primalfeas$. We have $x+y\in C + \relint{C}$ and $b - \A[(x+y)] = b-\A[x] -\A[y] \in \K + \relint{\K}$. This leads to $x+y\in\strict{X(b)}$ because of the fundamental fact about convex cones that their relative interiors are invariant to addition by the cone (i.e., a closed convex cone $\C$ obeys $\relint{\C}+\C=\relint{\C}$).

%\begin{claim}	\thlabel{relintprop2}
%A closed convex cone $\C$ obeys $\relint{\C}+\C=\relint{\C}$.
%\end{claim}
%
%This is proved in Appendix~\ref{sec:missing} for sake of completeness. 

For the inclusion of span of $\K$, we use the following technical result which says that existence of preimage of relative interior of a cone under a linear map implies that the preimage exists for all affine maps that are a suitable shift of the linear map.

\begin{claim}	\thlabel{relintaff}
Let $\C$ be a closed convex cone. 
If a linear map $\L$ has $\L^{-1}(\relint{\C})\neq\emptyset$, then $\G_{v}^{-1}(\relint{\C})\neq\emptyset$ for the affine map $\G_{v}(\cdot) = v + \L(\cdot)$ with $v\in\aff{\C}$. On the contrary, if an affine map $\G$ has $\G^{-1}(\relint{\C})\neq\emptyset$, then $\G_{\lambda}^{-1}(\relint{\C})\neq\emptyset$ for all affine maps $\G_{\lambda}(\cdot) = (1-\lambda)\G(\zeros)+\lambda\G(\cdot)$ where $\lambda \neq 0$.
\end{claim}
\claimproof{
We consider two cases for the first claim. First, suppose that $\L^{-1}(\relint{\C})\subseteq\kernel\L$. Since $\kernel\L = \L^{-1}(\{\zeros\})$, we have $\zeros\in\relint{\C}$ in this case. It follows from \mythref{relintprop3} that $\C$ must be a linear subspace. This implies $\aff{\C} = \relint{\C}$, and so for every $x\in\L^{-1}(\relint{\C})$ and $v\in\aff{\C}$ we have $\G_{v}x = v + \L x = v+\zeros = v \in \aff{\C} = \relint{\C}$, implying that $\G_{v}^{-1}(\relint{\C})\supseteq\L^{-1}(\relint{\C})\neq\emptyset$. Now suppose that $\L^{-1}(\relint{\C})\nsubseteq \kernel\L$. Therefore, there exists $x\in\L^{-1}(\relint{\C})$ such that $\L x \neq\zeros$. Take any $v\in\aff{\C}$. Then, $\L x\in\relint{\C}$ implies that there exists some $\epsilon > 0$ for which $\L x + \epsilon v \in\relint{\C}$. Scaling by $\epsilon$ and using the fact that $\relint{\C}$ is an open convex cone, gives us $\L(x/\epsilon) + v \in \relint{\C}$. %If $\kernel\L\supseteq\L^{-1}(\relint{\C})$, then $x\in\kernel\L$ and so $x/\epsilon\in\kernel\L$
Hence, $\G_{v}(x/\epsilon) \in \relint{\C}$, which implies that $x/\epsilon \in \G_{v}^{-1}(\relint{\C})$. 

For the second claim, take any $x\in\G^{-1}(\relint{\C})$. Denote $\L_{\G}(\cdot) = \G \zeros - \G(\cdot)$, which is a linear map. Then, $\G_{\lambda}(\cdot) = \G \zeros - \lambda\L_{\G}(\cdot)$. Because $\G x = \G \zeros - \L_{\G} x \in \relint{\C}$, linearity of $\L_{\G}$ implies that for any $\lambda \neq 0$, we have $\G x = \G \zeros - \lambda\L_{\G}(x/\lambda) = \G_{\lambda}(x/\lambda)$, and so $\G_{\lambda}^{-1}(\relint{\C})\neq\emptyset$.
}

$X(\zeros)$ is the preimage of the cone $\K\times C$ under the linear map $\L\colon x\mapsto (-\A[x],x)$, and $X(b)$ is the preimage of $\K\times C$ under the affine map $(b,\zeros) + \L(x)$. The definition of strict feasibility is that $\strict{X(\zeros)} = \L^{-1}(\relint{\K}\times\relint{C})$. The claim $\strict{X(b)} \neq\emptyset$ follows from \mythref{relintaff} due to $(b,\zeros) \in \linhull{\K}\times\linhull{C}$.
\end{proof}

Strictly feasible solutions may not always exist for the cone $X(\zeros)$; we can apply the  conditions presented earlier in this section with $b=\zeros$ to certify when $\strict X(\zeros)\neq\emptyset$. Note that if $\strict X(\zeros)=\{\zeros\}$, then it means that for each of $\K$ and $C$, either the cone is equal to $\{\zeros\}$ or is a subspace, but even in this case the above proof goes through because we have the cones equal to their relative interiors.

\subsection{Boundedness CQ} \label{sec:boundcond}

For a general convex program
\begin{equation}	\label{convexopt}
\inf\left\{f_{0}(x) \sep f_{i}(x)\le 0 \  i=1,\dots,m, \ x\in S \right\},
\end{equation}
\citet[Theorems 1 and 2]{duffin1978clark} established that if the primal has a bounded feasible set then the Lagrangian dual of \eqref{convexopt} has an unbounded feasible set and strong duality holds. We term this Clark-Duffin theorem as the ``Boundedness CQ'' for strong duality; see \citet{ernst2013zero} for generalisations of it wherein CQs are given in terms of recession directions satisfying certain properties. Since the conic program~\eqref{def:primal} can be reformulated in the form ~\eqref{convexopt} as $\inf_{x,s}\{c^{\top}x \colon \A[x] + s = b, \ x\in C, s\in\K\}$, and it is easy to see that the set $\widehat{X} := \{(x,s)\in C\times\K \colon \A[x] +s = b \}$ is bounded if and only if its projection $X$ is bounded, we have the Boundedness CQ for conic programs. Nevertheless, this CQ can also be argued independently for conic programs. \citet[Theorem~3.15]{schurr2007universal} did this using their results on universal duality for the case of $\K=\{\zeros\}$ (linear equalities in primal) and  $C$ being full-dimensional and pointed. We argue for general cones using conic Gordan's theorem of the alternative.

\begin{corollary}[Boundedness CQ]	\thlabel{boundedstrong}
%If the primal has a nonempty and bounded feasible region and $c\in(\lineal{C})^{\bot}$ (resp. $b\in\linhull{\K}$), then strong duality holds and the primal is solvable.
%When $C$ is pointed (resp. $\K$ is full-dimensional), boundedness of the primal (resp. dual) feasible region is a CQ for strong duality.
Suppose that the primal feasible set $X$ is nonempty and bounded. Whenever the dual is feasible, and in particular when $c\in (\lineal{C})^{\bot}$ or when $C$ is pointed, 
we have that strong duality holds; if we also assume that at least one of $\K$ or $C$ is not a subspace and not equal to $\{\zeros\}$, then the dual feasible set $Y$ is unbounded.
\end{corollary}
\begin{proof}
%The primal is obviously solvable due to compactness of $X$. 
%We deduce the primal condition that $c\in(\lineal{C})^{\bot}$ and boundedness of $X$ guarantee strong duality, and so when $C$ is pointed there is no restriction on $c$ since $(\lineal{C})^{\bot}$ becomes equal to $\E$.  
The two statements in \mythref{gordan} are $X(\zeros)\neq\{\zeros\}$ and $\strict{Y(\zeros)} \neq \emptyset$, because the definition in \eqref{strictfeas} tells us the strictly feasible solutions in the dual homogenous set are
\[
\strict{Y(\zeros)} \eq \left\{y\in\relint{\dualcone{\K}}\sep \mapAad[y] \coneg[\dualcone{C}] \zeros \right\}.
\] 
Hence, when $X$ is bounded, the first claim in this lemma tells us that $\strict{Y(\zeros)} \neq \emptyset$. The dual analogue of \mythref{strictrecXnot} gives us strict feasibility of $Y$ whenever $Y\neq\emptyset$, and in particular when $c \in \linhull{\dualcone{C}} = (\lineal{C})^{\bot}$,  then strong duality follows from the dual Slater CQ in \mythref{slaterstrong}. A pointed $C$ has $\dualcone{C}$ being full-dimensional (i.e., $\linhull{\dualcone{C}} =\E$) and so the dual is feasible for all $c\in\E$. Now let us argue unboundedness of $Y$. Non-subspace assumption on $\K$ or $C$ means that at least one of the dual cones is not a subspace. \mythref{relintprop3} implies that $\strict{Y(\zeros)}$ has a nonzero element, and then $Y(\zeros) = \rec{Y}$ from the dual analogue of \mythref{recconeX} leads to $Y$ being an unbounded set.
\end{proof}

Now we discuss sufficient conditions for the primal set to be bounded. Since $\A[(X(\zeros))] = \A(C)\cap-\K$ and $\rec{X} = X(\zeros)$ from \mythref{recconeX}, we know that $\A(C)\cap-\K=\{\zeros\}$ is a necessary condition for boundedness of the primal set. We argue that this condition is also sufficient when $\A$ is injective and $\K$ is pointed. It is not sufficient when only $C$ is pointed, which is the other condition that can make $X$ a pointed set, % for \mythref{gordan}, 
because we could have $\kernel\A\cap C\neq\emptyset$ which would imply existence of a nonzero point in $\rec{X}$ and therefore, unboundedness of $X$.

\begin{proposition}	\thlabel{packbound}
Suppose $X$ is feasible and $\A$ is injective. Then, $X$ is bounded if and only if $\A(C)\cap-\K=\{\zeros\}$. In particular, $X$ is bounded if $\A(C)\subseteq\K$ and $\K$ is pointed, or $\A(C)\subseteq\dualcone{\K}$.
\end{proposition}
\begin{proof}
The only if part is from \mythref{gordan}. For the if part, \mythref{recconeX} tells us that it suffices to show $\A(C\setminus\{\zeros\})\cap-\K=\emptyset$. Suppose there exists a nonzero $x \in C$ for which $\A[x]\in-\K$. Set $y=\A[x]$. Then $y\in\A(C)\cap-\K$. Because $\A(C)\cap-\K=\{\zeros\}$, it must be that $y=\zeros$ and therefore $x\in\kernel\mapA$. The assumption that $\A$ is injective is equivalent to $\kernel\A = \{\zeros\}$, and so $x=\zeros$, which is a contradiction. Therefore, $\XX$ must be bounded. The particular conditions each imply $\A(C)\cap-\K=\{\zeros\}$ because $\K\cap-\K=\{\zeros\}$ for a pointed $\K$ and any closed convex cone $\C$ satisfies $\dualcone{\C}\cap-\C=\{\zeros\}$.
\end{proof}

An unbounded set may or may not have strictly feasible solutions to its recession cone. A pointed set is unbounded if and only if the recession cone of the other problem does not have any strictly feasible solutions. %\hyperref[gordan]{\thref{gordan}} established the necessity part without requiring $C$ to be pointed. 
We also present a sufficient condition for boundedness in terms of a basis for the span of $C$.

\begin{proposition}	\thlabel{boundbasis}
Let $X\neq\emptyset$ be pointed. It is bounded if and only if $\strict{Y(\zeros)}\neq\emptyset$.

Furthermore, it is bounded if there exists an orthonormal basis $B =\{v^{1},\dots,v^{m}\}$ of $\linhull{C}$ such that $B\subseteq\dualcone{C}$ and $\A(B)\subseteq\K$ with $v^{j}\in\relint{\dualcone{C}}$ and $\A[v^{j}]\in\K\setminus\lineal{\K}$ for some $j$. 
%\item $\dualcone{\K}\neq\K^{\bot}$ and $A_{j} \in \relint{\K}$ for all $j$, or
\end{proposition}
\begin{proof}
Pointedness of $X$ means that $\rec{X}\cap-\rec{X} = \{\zeros\}$, and since $\rec{X}=X(\zeros)=-\mapA^{-1}(\K)\cap C$, it follows that either $X(\zeros)=\{\zeros\}$ or at least one of $\mapA^{-1}(\K)$ or $C$ is not a subspace. Then, the characterisation for boundedness follows directly from \mythref{gordan}.

\mythref{dualbasis} tells us that we can replace the adjoint map in the dual with the linear map $\mapA[B]\colon y \mapsto \sum_{j=1}^{m}\inner{\A[v^{j}]}{y}v^{j}$. Therefore, the recession cone of the dual is $\{y\in\dualcone{\K}\colon \mapA[B]y \conegeq[\dualcone{C}]\zeros \}$. \mythref{gordan} applied to this recession cone tells us that to show that $X$ is bounded, it suffices to show that there exists $y\in\relint{\dualcone{\K}}$ such that $\mapA[B]y \in\relint{\dualcone{C}}$. The assumption $\A[v^{j}]\in\K\setminus\lineal{\K}$ implies that $\K\setminus\lineal{\K}\neq\emptyset$, which is equivalent to $\K$ not being a linear subspace. Then \mythref{relintprop3} implies that there exists a nonzero $y\in\relint{\dualcone{\K}}$. The second claim in this lemma combined with the assumption $\A[v^{j}]\in\K\setminus\lineal{\K}$ gives us $\inner{\A[v^{j}]}{y} > 0$. This leads to $\inner{\A[v^{j}]}{y}v^{j}\in\relint{\dualcone{C}}$ due to $v^{j}\in\relint{\dualcone{C}}$ and the fact that $\relint{\dualcone{C}}$ is a cone. The assumptions $v^{i}\in\dualcone{C}$ and $\A[v^{i}]\in\K$ for $i\neq j$ give us $\inner{\A[v^{i}]}{y} \ge 0$ and $\inner{\A[v^{i}]}{y}v^{i}\in\dualcone{C}$, thereby leading to $\sum_{i\neq j}\inner{\A[v^{i}]}{y}v^{i}\in\dualcone{C}$. Since $\mapA[B]y =\inner{\A[v^{j}]}{y}v^{j} + \sum_{i\neq j}\inner{\A[v^{i}]}{y}v^{i}$, and the relative interior of a cone is invariant to addition with the cone, %(cf.~\mythref{relintprop2} in the proof of \mythref{strictrecXnot}), 
we obtain the desired claim $\mapA[B]y \in \relint{\dualcone{C}}$.
%Denote $d := \sum_{i\neq j^{*}}\inner{\A[v^{i}]}{\bar{y}}v^{i}$ and note that $d \in \aff{\dualcone{C}}$ due to $v^{i}\in\aff{\dualcone{C}}$ for all $i$. Then, 
%%$\mapAad_{C}y = \inner{\A[v^{j}]}{y}v^{j} + \sum_{i\neq j}\inner{\A[v^{i}]}{y}v^{i}$ and $v^{i}\in\aff{\dualcone{C}}$ for all $i$, 
%$\inner{\A[v^{j}]}{\bar{y}}v^{j}\in\relint{\dualcone{C}}$ implies that there exists some $\epsilon > 0$ such that $\inner{\A[v^{j}]}{\bar{y}}v^{j^{*}} + \epsilon d \in\relint{\dualcone{C}}$ \akshay{incomplete here}
%$\dualcone{\K}\neq\K^{\bot}$ is equivalent to $\dualcone{\K}\setminus\K^{\bot}\neq\emptyset$ since $\C^{\bot}=\dualcone{\C}\cap\polar{\C}$ for any cone $\C$. For any two nonempty closed convex sets $\X_{1}$ and $\X_{2}$, $\X_{1}\setminus\X_{2}\neq\emptyset$ is equivalent to $\relint{\X_{1}}\setminus\X_{2}\neq\emptyset$ because $\closure{(\relint{\X_{1}})} = \X_{1}$ and $\closure{\X_{2}}=\X_{2}$. Hence, our condition is equivalent to $\relint{\dualcone{\K}} \setminus \K^{\bot} \neq\emptyset$. Take any $y\in \relint{\dualcone{\K}} \setminus \K^{\bot}$. \hyperref[relintprop]{\thref{relintprop}} and $A_{j} \in \relint{\K}$ gives us $\inner{A_{j}}{y} > 0$, for all $j$. This means $\mapAad[y] \in \relint{\dualcone{C}}$, and then \hyperref[gordan]{\thref{gordan}} implies $X$ is bounded.
\end{proof}

\section{A Theorem of the Alternative on Almost Feasibility}	\label{sec:almostfeas}

Besides strict feasibility, which guarantees existence of Slater CQ, another notion on feasibility is that of \emph{almost feasibility} which is defined for $X(b)$ with respect to some chosen norm $\|\cdot\|$ in $\EE$ as follows:
\begin{equation}	\label{def:almost}
\text{(Almost feasibility of $X$)} \sep \forall\, \epsilon > 0, \ \exists\, b^{\epsilon}\in\EE \ \text{ s.t. } \|b^{\epsilon}\| \le \epsilon \ \text{ and } \ X(b + b^{\epsilon}) \neq \emptyset.
\end{equation}
The definition for $Y(c)$ is analogous. Our definitions are inspired by \citet[\textsection 1.7.2]{bentalnem2001} who call it \emph{almost solvability}. A set that is infeasible but almost feasible is called \emph{weakly infeasible} by \citet{luo1997duality}. Algorithmic aspects and other issues surrounding almost feasibility are discussed by \citet{polik2009new,liu2017exact}. Clearly,
\begin{center}
Strict feasibility $\implies$ Feasibility $\implies$ Almost feasibility.
\end{center}

Our main result here is to show that under some conditions, almost feasibility of the dual is equivalent to membership in the  polar of the recession cone of the primal. 

\begin{theorem}		\thlabel{polaralmost}
$\polar{(X(\zeros))} = \{c\in\E\colon Y(c) \text{ is almost feasible} \}$ when either $\K$ is a subspace or $\mapAad(\relint{\dualcone{\K}})\cap(\lineal{C})^{\bot}\neq\emptyset$.
\end{theorem}

%Special cases of this result include $C$ being pointed, which implies the equality for the polar of $\rec{X}$, and $\K$ being  full-dimensional, which implies the equality for the polar of $\rec{Y}$. 
%Note that if $C$ is pointed, then the first equality holds, and if $\K$ is full-dimensional then the second one holds. 

This result can be interpreted as a theorem of the alternative because the assertion tells us that for every $c\in\E$, we have exactly one of the following two statements being true:
\begin{enumerate}
\item either $Y(c)$ is almost feasible, or
\item $\inner{c}{x} > 0$ for some $x\in C$ with $\A[x]\coneleq \zeros$.
\end{enumerate}

\begin{remark}
The analogous version of \mythref{polaralmost} for the dual problem is that
\begin{equation}	\label{dualpolaralmost}
\dualcone{(Y(\zeros))} \eq \{b\in\EE\colon X(b) \text{ is almost feasible} \}
\end{equation}
when either $C$ is a subspace or $\A(\relint{C})\cap\linhull{\K}\neq\emptyset$. This is a generalization of \citet[Proposition 1.7.1]{bentalnem2001} which 
%can be restated as $b\in\dualcone{Y(\zeros)}$ if and only if $X(b)$ is almost feasible, and 
deals with $C=\E$ and a full-dimensional $\K$ and is therefore a special case of our result because it satisfies both of our conditions\footnote{Since \citep{bentalnem2001} deal with the primal as a minimisation problem, the obvious modifications have to be made when comparing their results to ours.}.
\end{remark}

Before proving our theorem, we derive a consequence of it which is the equivalence of feasibility and almost feasibility for bounded sets, and infeasibility of one problem implying an unbounded optimum for the other problem. 

\begin{corollary}	\thlabel{almostequivfeas}
Assume $X(\zeros) = \{\zeros\}$ and either $C$ is a subspace or $\A(\relint{C})\cap\linhull{\K}\neq\emptyset$. The following are equivalent for any $b\in\EE$, 
\begin{enumerate}
\item $b\in\dualcone{(Y(\zeros))}$,
\item $X(b)$ is feasible,
\item $X(b)$ is almost feasible.
\end{enumerate}
Furthermore, if we also have that the primal is infeasible, then $\optdual = -\infty$ when $c\in(\lineal{C})^{\bot}$, otherwise the dual is either infeasible or has unbounded optimum when $c\notin(\lineal{C})^{\bot}$.
\end{corollary}
\begin{proof}
The first and third are equivalent by \eqref{dualpolaralmost}, and second implying the third is  trivially true without any assumptions. Take any $b\in\EE$ is such that $X(b)$ is almost feasible. \mythref{gordan} tells us that $X(\zeros) = \{\zeros\}$ implies $\strict{Y(\zeros)} \neq \emptyset$. The dual analogue of the last part of \mythref{recconeX} gives us $\dualcone{(Y(\zeros))} = \primalfeas$, and then from \eqref{dualpolaralmost} it follows that $b\in\primalfeas$, and so $X(b)$ is feasible.

Now suppose that the primal is infeasible. The first part of this proof tells us that $b\notin\dualcone{(Y(\zeros))}$, which means that there exists some $r\in Y(\zeros)$ such that $\inner{b}{r} < 0$. Since the dual is a minimisation, when it is feasible, the ray $r$ leads to an optimum of $-\infty$. By \mythref{gordan}, $\strict{Y(\zeros)} \neq \emptyset$ and then the dual analogue of \mythref{strictrecXnot} asserts dual feasibility for $c\in\linhull{\dualcone{C}} = (\lineal{C})^{\bot}$. %the dual is feasible when
\end{proof}

Without the above conditions it is possible to have infeasibility for one and finite optimum for the other as seen in \mythref{exampleAdapted} which has $X(\zeros)\neq\{\zeros\}$.

Now we proceed to prove \mythref{polaralmost} by first establishing some technical results.

\subsection{Lemmata}

The set of all right-hand sides for which the dual problem is almost feasible is sandwiched between the polar cone of $X(\zeros)$ and the relative interior of this polar, and therefore one direction of inclusion in \mythref{polaralmost} is always true without requiring any assumptions.

\begin{lemma}	\thlabel{almostfeas1}
$\polar{(X(\zeros))} \,\supseteq\, \{c\in\E\colon Y(c) \text{ is almost feasible}\} \,\supseteq\, \relint{\polar{(X(\zeros))}}$.
\end{lemma}
\begin{proof}
The second inclusion is due to $\relint{\polar{(X(\zeros))}} = \relint{\dualfeas}$ from \mythref{recconeX} and the fact that strict feasibility implies feasibility which implies almost feasibility. Let us prove the first inclusion by contraposition. Let $c\notin\polar{(X(\zeros))}$. Then there exists some nonzero $r\in X(\zeros)$ for which $\inner{c}{r} > 0$. Pick any norm $\|\cdot\|$ in $\EE$ and set $\delta = \inner{c}{r}$ and $\epsilon = \frac{\delta}{2\|r\|}$. For any $c^{\epsilon}\in\EE$ such that $\|c^{\epsilon}\| \le \epsilon$, we have
\[
\inner{c+c^{\epsilon}}{r} \,\ge\, \inner{c}{r} - \abs{\inner{c^{\epsilon}}{r}} \,\ge\, \inner{c}{r} - \|c^{\epsilon} \| \| r\|  \,\ge\, \inner{c}{r} - \epsilon \|r\| \eq \delta - \frac{\delta}{2} \eq \frac{\delta}{2} \,>\, 0,
\]
where the second inequality is by the Cauchy-Schwarz inequality. Therefore, $c+c^{\epsilon} \notin \polar{(X(\zeros))}$, and then $\polar{(X(\zeros))}\supseteq\dualfeas$ from \mythref{recconeX} implies $c+c^{\epsilon} \notin\dualfeas$, which means that $Y(c+c^{\epsilon}) = \emptyset$. Since $c^{\epsilon}$ was arbitrary up to $\|c^{\epsilon}\| \le \epsilon$, it follows from the definition in \eqref{def:almost} that $Y(c)$ is not almost feasible.
\end{proof}

%First, we have that elements of the polar cone that do not permit feasibility of an auxiliary dual problem do permit almost feasibility of the original dual. 
Define
\[
\Omega \define \left\{c\in\E\colon \nexists y\in\dualcone{\K}, w\in\dualcone{\K}\setminus\lineal{\dualcone{\K}} \text{ s.t. } \mapAad[(y-w)] \conegeq[\dualcone{C}] c\right\}.
\]
We will need a technical lemma on $\Omega$ for proving \mythref{polaralmost}. 
Let us recall the fact that relative interior of a non-subspace convex cone is exactly the points in its span that have a positive product with every point in the dual cone that is not in the orthogonal complement.

\begin{lemma}[cf. {\cite[Theorem 2]{luo1997duality}}]	\thlabel{relintprop}
Suppose $\C$ is a closed convex cone that is not a linear subspace.\footnote{The assertion here also works when $\C$ is a linear subspace, if we interpret the product being positive to be a vacuously true statement and drop it so that $\relint{\C}$ becomes equal to $\aff{\C}$, which is obvious for a subspace.}  
We have $\relint{\C} = \{x\in\aff{\C}\sep \inner{x}{y} > 0, \;\; \forall y \in \dualcone{\C}\setminus\C^{\bot} \}$, and so $\relint{\C} \not\perp \dualcone{\C}\setminus\C^{\bot}$.

Similarly, $\relint{\dualcone{\C}} = \{y \in (\lineal{\C})^{\bot} \sep \inner{y}{x} > 0, \;\; \forall x \in \C\setminus\lineal{\C} \}$ and $\relint{\dualcone{\C}} \not\perp \C\setminus\lineal{\C}$.
\end{lemma}

\begin{lemma}	\thlabel{lemalmost}
$\polar{(\rec{X})}\cap\Omega = \{c\in\Omega\mid Y(c) \text{ is almost feasible}\}$.
\end{lemma}
\begin{proof}
The $\supseteq$ inclusion is from \mythref{almostfeas1}. Take $c\in\Omega$ and suppose $Y(c)$ is not almost feasible. We have to prove that $c\notin\polar{(\rec{X})}$. Choose any $\xi\in\relint{\dualcone{C}}$ and $\gamma\in\relint{\K}$ and consider the conic problem 
\[
z^{\ast} \eq \inf\left\{t_{1}+t_{2}+\inner{\gamma}{w} \sep \mapAad[y] - \mapAad[w] + t_{1}c + t_{2}\xi \conegeq[\dualcone{C}] c, \ y\in\dualcone{\K}, w\in \dualcone{\K}, t_{1},t_{2}\ge 0 \right\}.
\] 
Denote the feasible set by $Y_{\xi}$. Observe that $(\bar{y},\bar{y},1,1)\in\strict{Y_{\xi}}$ for any $\bar{y}\in\relint{\dualcone{\K}}$. The dual problem to $z^{\ast}$ is 
\[
\sup\left\{\inner{c}{x}\sep \A[x]\coneleq\zeros, -\A[x]\coneleq\gamma, \inner{c}{x} \le 1, \inner{\xi}{x} \le 1, x\in C\right\}.
\] 
This has a feasible solution $x=\zeros$ due to $\gamma\in\relint{\K}$. Then strong duality from \mythref{slaterstrong} implies that there exists a feasible $x^{\ast}$ to the dual problem with $\inner{c}{x^{\ast}} = z^{\ast}$. It is clear that $z^{\ast}\ge 0$. We claim that $z^{\ast} > 0$. Because the feasible set of the dual problem is a subset of $\{x\in C\colon \A[x]\coneleq\zeros\} = \rec{X}$, we get $x^{\ast}\in\rec{X}$, and so $z^{\ast} > 0$ implies $c\notin\polar{(\rec{X})}$, thereby finishing our proof.

Suppose for the sake of contradiction that $z^{\ast} = 0$. By \mythref{relintprop} and $\gamma\in\relint{\K}$, the term $\inner{\gamma}{w}$ is equal to zero for any $w\in\dualcone{\K}$ if and only if $w\in\K^{\bot}$. \mythref{spandual} with $\C=\dualcone{\K}$ and the Bipolar Theorem gives us $\K^{\bot} = \lineal{\dualcone{\K}}$. Because $(\bar{y},\bar{y},1,1)\in\strict{Y_{\xi}}$ for any $\bar{y}\in\relint{\dualcone{\K}}$, we know that the objective function is not identically equal to zero over $Y_{\xi}$. Because the infimum is equal to zero, for arbitrarily small $\epsilon > 0$ there exist feasible solutions $(y,w,t)\in Y_{\xi}$ with $t_{1}+t_{2}+\inner{\gamma}{w} = \epsilon$. Then one of the following two things must happen: (i) at least one of $t_{1}$ or $t_{2}$ is strictly positive and $w\in\lineal{\dualcone{\K}}$, or (ii) $t_{1}=t_{2}=0$ and $w\in\dualcone{\K}\setminus\lineal{\dualcone{\K}}$. However, $c\in\Omega$ means that the second case is not possible. In the first case, set $\delta = \epsilon(\|c\|+\|\xi\|)$ and $c^{\delta} = -t_{1}c-t_{2}\xi$. Because $t_{1}, t_{2} \le \epsilon$, we have $\|c^{\delta}\| \le \delta$. Because $\epsilon$ is arbitrary close to zero, $\delta$ is also arbitrary close to zero. Then $\mapAad[(y-w)]\conegeq[\dualcone{C}] c + c^{\delta}$ and $y-w \in \dualcone{\K}$ due to $y\in\dualcone{\K},w\in\lineal{\dualcone{\K}}$ implies that for every $\delta > 0$, we have $y-w\in Y(c+c^{\delta})$, thereby making $Y(c)$ almost feasible, which is a contradiction. 
\end{proof}

We will need the following geometric property of cones that one can enter (resp. exit) a cone in some suitable directions starting from points outside (resp. inside) the cone; this is proved in Appendix~\ref{sec:missing} for completeness.

\begin{lemma}	\thlabel{entercone}
Let $\C$ be a nonempty closed convex cone and $x\in\C$ and $y\in\linhull{\C}$. 
\begin{enumerate}
\item If $y\notin\C$, there exists a finite $t^{\prime} \ge 0$ such that $x+ty\in\C$ for all $0 \le t\le t^{\prime}$ and $x+ty\notin\C$ for all $t > t^{\prime}$. %Furthermore, $t^{\prime}=0$ if $y\notin\aff{\C}$, and if $y\in\aff{\C}$ then
Also, $t^{\prime}=0$ only if $x\in\boundr{\C}$.
\end{enumerate}
Now suppose $x\in\relint{\C}$.
\begin{enumerate}[resume]
\item There exists a finite $\delta^{\star}> 0$ such that $y+\delta x\in\relint{\C}$ for all $\delta>\delta^{\star}$. 
\item If $y\notin\C$, then $y+\delta^{\star} x\in\bd{\C}$ and $y+\delta x \notin\C$ for all $0 \le \delta < \delta^{\star}$.
\end{enumerate}
\end{lemma}

\subsection{Proof of \hyperref[polaralmost]{Theorem~\ref{polaralmost}}}

%\begin{proof}[\textbf{Proof of \mythref{polaralmost}}]
The dual cone of a subspace is also a subspace and so $\K$ being a subspace implies $\dualcone{\K} = \lineal{\dualcone{\K}}$, which makes $\Omega=\E$ in \mythref{lemalmost}. The equality for the polar cone follows from the lemma.

Now let $\mapAad(\relint{\dualcone{\K}})\cap\aff{\dualcone{C}}\neq\emptyset$. The $\supseteq$ inclusion is from \mythref{recconeX}. We argue the $\subseteq$ inclusion by contraposition. Suppose $Y(c)$ is not almost feasible. Choose any $\xi\in\relint{\dualcone{C}}$ and consider the conic problem 
\[
z^{\ast} \eq \inf\left\{t_{1}+t_{2} \sep \mapAad[y] + t_{1}c + t_{2}\xi \conegeq[\dualcone{C}] c, \ y \in \dualcone{\K}, t \ge\zeros \right\}.
\] 
Denote the feasible set $Y_{\xi}$. The dual problem to $z^{\ast}$ is 
\[
\sup\left\{\inner{c}{x}\sep \A[x]\coneleq\zeros, \inner{c}{x} \le 1, \inner{\xi}{x} \le 1, x\in C\right\}.
\] 
This has a feasible solution $x=0$. By assumption, there exists some $\bar{y}\in\relint{\dualcone{\K}}$ with $\mapAad[\bar{y}]\in\aff{\dualcone{C}}$. \mythref{entercone} with $\C=\dualcone{C}$ and $\xi\in\relint{\dualcone{C}}$ imply that $(\bar{y},1,t_{2}) \in \strict{Y_{\xi}}$ for some $t_{2} > 0$. Then strong duality under Slater CQ from \mythref{slaterstrong} implies that there exists a feasible $x^{\ast}$ to the dual problem with $\inner{c}{x^{\ast}} = z^{\ast}$. It is clear that $z^{\ast}\ge 0$. We claim that $z^{\ast} > 0$. Because the feasible set of the dual problem is a subset of $\{x\in C\colon \A[x]\coneleq\zeros\} = \rec{X}$, we get $x^{\ast}\in\rec{X}$, and so $z^{\ast} > 0$ implies that $c\notin\polar{(\rec{X})}$, which finishes our proof by contraposition for the $\subseteq$ inclusion. To argue the claim $z^{\ast} > 0$, suppose that $z^{\ast}=0$. This means that there exist feasible solutions to $Y_{\xi}$ with both $t_{1}$ and $t_{2}$ arbitrarily close to zero. Then for any $\epsilon>0$, we can choose $c^{\epsilon} = -t_{1}c - t_{2}\xi$ for small values of $t_{1}$ and $t_{2}$ to get $\|c^{\epsilon}\|\le\epsilon$, and we would have some $y\in\dualcone{\K}$ such that $\mapAad[y]\conegeq[\dualcone{C}] c+c^{\epsilon}$. But this leads to the contradiction that $Y(c)$ is almost feasible.
\hfill\qedsymbol

\section{Projecting a Conic Set onto a Subspace}	\label{sec:proj}

A well-known result for linear programming is that the orthogonal projection of a polyhedron onto a subspace of variables can be obtained by multiplying the algebraic linear description of the polyhedron by every extreme ray of a specific cone called the projection cone (cf.~\citep[Theorem 1.1]{balas2005projection}). This can be proven using LP strong duality and it is very useful for deriving valid inequalities to mixed-integer programs. It is natural to expect that a similar result holds for sets described using conic inequalities because conic programs also exhibit strong duality, albeit under stronger assumptions than those required for linear programming. We describe the projection of conic sets onto arbitrary linear subspaces by giving a straightforward proof that is motivated by the polyhedral case. Note that our question is different than that of projecting a point onto a conic set, for which algorithms are given by \citet{henrion2012projection}.

Consider the conic set $X = \{x\in C\colon\A[x]\coneleq b\}$ from~\eqref{def:primal} and let $L\subset\E$ be a linear subspace. The projection of $X$ onto $L$ is defined as \[\proj_{L}X \define \left\{x\in L\sep \exists u\in L^{\bot} \text{ s.t. }  x+u\in X\right\}.\] Denote 
\[
\mathscr{C} \define \left\{(y,w)\in\dualcone{\K}\times\dualcone{C}\sep \mapAad[y]-w \in L\right\}.
\]
Clearly, this is a convex cone. It is also a closed set because it is equal to the intersection of two closed sets --- (i) $\dualcone{\K}\times\dualcone{C}$, and (ii) the preimage of the subspace $L$ under the linear map $(y,w)\mapsto \mapAad[y]-w$. The latter is a closed set because the preimage of any closed set under a continuous map is a closed set. Thus, $\mathscr{C}$ is a closed convex cone and so it is generated by its extreme rays.

\begin{theorem}	\thlabel{projection}
Suppose there exists $\hat{y}\in\relint{\dualcone{\K}}$ and $\hat{w}\in\relint{\dualcone{C}}$ such that $\mapAad[\hat{y}]-\hat{w} \in L$. Then,
\[
\proj_{L}X \eq \left\{x\in L \sep \inner{\mapAad[y]-w}{x}[\E] \,\le\, \inner{b}{y}[\EE], \ \forall \text{ extreme rays } (y,w) \in \mathscr{C} \right\}.
\]
\end{theorem}
\begin{proof}
The definition of projection tells us that
\begin{equation}	\label{existu}
\proj_{L}X \eq \left\{x\in L \sep \exists u\in L^{\bot} \text{ s.t. } \A[u] \coneleq b-\A[x], \ u\in C-x \right\}.
\end{equation}
Consider the conic program
\[
\sup_{u}\left\{\inner{\zeros}{u} \sep \A[u] \coneleq b-\A[x], \ u\in C-x, \  u\in L^{\bot}\right\}.
\]
Using the fact that $(L^{\bot})^{\bot} = L$, the dual of the above primal can be easily derived to be
\[
\inf_{y,w}\left\{\inner{b-\A[x]}{y} + \inner{x}{w} \sep (y,w) \in \mathscr{C}\right\}.
\]
The existence of $(\hat{y},\hat{w})$ guarantees strict feasibility of $\mathscr{C}$ because $\relint{L}=L$ for a subspace $L$, and so the dual problem has Slater CQ, and then \mythref{slaterstrong} implies that we have strong duality. Because the dual optimizes over a cone, its value is either 0 or $-\infty$, and so strong duality implies that the dual value is 0 if and only if the primal is feasible. Hence, the existence of a $u$ in equation~\eqref{existu} can be equivalently replaced with the infimum of the dual being zero, which is equivalent to $\inner{b-\A[x]}{y} + \inner{x}{w} \ge 0$ for all $(y,w) \in \mathscr{C}$. Because $\mathscr{C}$ is a cone, it suffices to enforce this nonnegativity for only its extreme rays. Substituting this in equation~\eqref{existu} yields
\[
\proj_{L}X \eq \left\{x\in L \sep \inner{b-\A[x]}{y} + \inner{x}{w} \ge 0, \ \forall \text{ extreme rays } (y,w) \in \mathscr{C} \right\}.
\]
The claimed expression follows after rearranging terms and using the adjoint.
\end{proof}

The cone $\mathscr{C}$ could have uncountably many extreme rays, in which case the projection is defined by an infinite system of linear inequalities. 

Applying \mythref{projection} to a high-dimensional set gives us its  orthogonal projection. 

\begin{corollary}
The orthogonal projection of $\X = \{(x,x')\in C\times C'\colon \A[x]+\mapA[B]x' \coneleq b \}$ onto the $x$-space is equal to
\[
\left\{x\in C\colon \inner{\mapAad[y]}{x} \le \inner{b}{y}, \, \forall \text{ extreme rays } y\in\mathscr{C} \right\},
\]
for the cone $\mathscr{C} = \{y\in\dualcone{\K}\colon \adjoint{\mapA[B]}y \conegeq[\dualcone{C'}] \zeros \}$, when there exists a $y\in\relint{\dualcone{\K}}$ with $\adjoint{\mapA[B]}y \in \relint{\dualcone{C'}}$.
\end{corollary}

\section{Conclusions}

Of the three CQs addressed in this paper, we showed that a Closedness CQ is the weakest condition since it implies the Slater CQ which in-turn implies the Boundedness CQ. We also gave a generalised version of the last CQ using a conic theorem of the alternative on strict feasibility. Another theorem of the alternative on almost feasibility was presented in its most general form. Subspace projections of conic sets were characterised using extreme rays of a projection cone, thereby extending the well-known idea for polyhedral sets. 

An interesting direction of future research would be to establish more connections between the different sufficient conditions for strong duality in literature, particularly for abstract convex programs. In particular, it would be good to know whether the Closedness CQ used in this paper is also necessary for strong duality, and how it relates to another known CQ  (cf.~\mythref{closednessrem}).
%This paper establishes various conditions for strong duality to hold between a primal-dual pair of conic programs. When the cones are low-dimensional, closedness of a particular adjoint image is sufficient for strong duality. This leads to many specific sufficient conditions that are more tractable to verify, such as generalized Slater constraint qualification, boundedness of the feasible region, and some that are in terms of the recession cones of primal or dual feasible regions. The paper also gives many necessary and sufficient conditions for having a bounded feasible region and then provides a conic theorem of the alternative  in terms of approximate feasibility. Finally, it shows that under generalized Slater constraint qualifications, finiteness of one problem and feasibility of the other are equivalent.

%An interesting direction of future research is to see whether the closedness condition and its many implications also hold for abstract convex programs that have been studied in literature with regards to a minimal facial property of cones. Another open question is to investigate which of the results in this paper extend to infinite-dimensional conic programs over general Hilbert spaces.

\subsubsection*{Acknowledgements}
During the initial part of this project, the first and fourth author (TA and AS) were supported by the National Science Foundation (USA) grant CMMI-1933373, and the second author (AG) was supported by the National Science Foundation (USA) grant DMS-1913294 when he was in the School of Mathematical \& Statistical Sciences at Clemson University.

{
\newrefcontext[sorting=nyt]
\printbibliography
}

\newpage

\begin{appendices}

\section{Missing Proofs}	\label{sec:missing} %on Technical Properties

\begin{proof}[\textbf{Proof of \mythref{genfeas}}]
The equalities for $\primalfeas$ and $\dualfeas$ are straightforward from their definitions and equations~\eqref{def:XY}. The relative interiors follow from basic properties of $\relint{}$ given in \mythref{facts}. Distributivity over sum gives us $\relint{\primalfeas} = \relint{\A(C)} + \relint{\K}$, and commutativity with a linear map leads to $\relint{\primalfeas} = \A(\relint{C}) + \relint{\K}$. The definition of strict feasibility makes it easy to see that $\strict{X(b)}\neq\emptyset$ if and only if $b\in\A(\relint{C}) + \relint{\K}$.
\end{proof}

\begin{proof}[\textbf{Proof of \mythref{relintprop3}}]
\mythref{relintprop} implies that $\relint{\dualcone{\C}} \cap \C^{\bot}= \emptyset$. Because $\C^{\bot}\subseteq \dualcone{\C}$ and $\dualcone{\C}$ is the disjoint union $\relint{\dualcone{\C}} \cup \boundr{\dualcone{\C}}$, it follows that $\C^{\bot}\subseteq\boundr{\dualcone{\C}}$. Substituting $\C$ with $\dualcone{\C}$ and using the Bipolar Theorem and \mythref{spandual} transforms $\C^{\bot}\subseteq\boundr{\dualcone{\C}}$ to $\lineal{\C}\subseteq\boundr{\C}$. The second claim follows from the first claim due to $\zeros\in\lineal{\C}\cap\C^{\bot}$, $\relint{\C}\cap\boundr{\C}=\emptyset$ and $\relint{\dualcone{\C}}\cap\boundr{\dualcone{\C}}=\emptyset$.
\end{proof}

\begin{proof}[\textbf{Proof of \mythref{gordan}}]
The first statement being false means that $X(\zeros)=\{\zeros\}$. Then, $\polar{(X(\zeros))} =\E$, and so \mythref{recconeX} gives us $\closure{\dualfeas} = \E$, which  implies that $\zeros\in\relint{(\closure{\dualfeas})} = \relint{\dualfeas}$. In fact, it is easy to argue that the second statement is equivalent to $\zeros\in\relint{\dualfeas}$. Hence, the second statement holds because of the description of $\relint{\dualfeas}$ from \mythref{genfeas}. If $X(\zeros)$ is not a subspace and second statement is true, we have $\zeros\in\relint{\dualfeas} = \relint{\polar{X(\zeros)}}$. \mythref{relintprop3} and the non-subspace assumption implies that $X(\zeros)$ must be equal to $\{\zeros\}$.
%To argue the reverse direction, we use the following technicality.
%
%\begin{claim}	\thlabel{equivstrictrec}
%The following statements are equivalent:
%\begin{enumerate}
%\item $\strict{Y(\zeros)} = \emptyset$,
%\item $\zeros\notin\relint{\dualfeas}$,
%\item there exists $\zeros\neq\lambda\in\E$ such that $\inner{\lambda}{y} \le 0 \le \inner{\lambda}{y'}$ for all $y\in \closure{\mapAad(\dualcone{\K})}$ and $y'\in \dualcone{C}$.
%\end{enumerate}
%\end{claim}
%\claimproof{
%(1 $\iff$ 2) $\strict{Y(\zeros)} = \emptyset$ if and only if $\mapAad(\relint{\dualcone{\K}}) \cap \relint{\dualcone{C}} =\emptyset$, which is equivalent to $\zeros\notin\mapAad(\relint{\dualcone{\K}}) - \relint{\dualcone{C}}$. \mythref{genfeas} implies that $\mapAad(\relint{\dualcone{\K}}) - \relint{\dualcone{C}} = \relint{\dualfeas}$.
%
%(1 $\iff$ 3) Because $\dualcone{\C} = \polar{(-\C)}$ for any convex cone $\C$, \mythref{polarinv} gives us $\dualcone{(\A^{-1}(\K))} = \closure{\mapAad(\dualcone{\K})}$, and so we have $\relint{\dualcone{(\A^{-1}(\K))}} = \relint{(\closure{\mapAad(\dualcone{\K})})} = \relint{\mapAad(\dualcone{\K})} = \mapAad(\relint{\dualcone{\K}})$. Therefore, statement (1) is equivalent to $\relint{\dualcone{(\A^{-1}(\K))}} \cap \relint{\dualcone{C}} = \emptyset$. The separation theorem states that an empty intersection of relative interiors is equivalent to the existence of a separator $\lambda$ satisfying the third statement. 
%}
\end{proof}

\begin{proof}[\textbf{Proof of \mythref{strictfeasrelint}}]
The only if direction is due to the feasibility of a convex set being equivalent to the feasibility of its relative interior. Now suppose $\strict{X}\neq\emptyset$. The strict conic inequality $\A[x]\conel b$ is defined as $b - \A[x] \in \relint{\K}$, which is equivalent to  $x\in\A^{-1}(b-\relint{\K})$. We will use the basic properties from \mythref{facts}. The map $\K \mapsto b-\K$ is affine and then commutativity of an affine map with $\relint{}$ implies that $b-\relint{\K} = \relint{(b-\K)}$, leading to $\A[x]\conel b$ if and only if $x\in \A^{-1}(\relint{(b-\K)})$. Thus, $\strict{X} = \A^{-1}(\relint{(b-\K)}) \cap\relint{C}$, and so $\A^{-1}(\relint{(b-\K)})$ is nonempty. Using the third claim from the lemma  gives us $\strict{X} = \relint{\A^{-1}(b-\K)}\cap\relint{C}$. Because $X = \A^{-1}(b-\K) \cap C$, the distributivity of $\relint{}$ over nonempty intersection implies that $\relint{X} = \relint{\A^{-1}(b-\K)}\cap\relint{C}$, giving us the desired equality $\relint{X} = \strict{X}$.
\end{proof}

\begin{proof}[\textbf{Proof of \mythref{entercone}}]
Suppose $y\notin\C$ and $x+ty\in\C$ for all $t \ge 0$. Because $\C$ is a cone, for $t>0$, we have $\frac{1}{t}(x+ty) = \frac{x}{t} + y \in\ C$. Because $\C$ is closed, $\lim_{t\to\infty}\frac{x}{t} + y \in \C$. But this limit is equal to $y$, giving a contradiction to $y\notin\C$. Hence, there exists some $\tilde{t} \ge 0$ for which $x+\tilde{t}y\notin\C$. Convexity of $\C$ and $x\in\C$ implies that $x+ty\in\C$ is not possible for any $t\ge \tilde{t}$. Taking $t^{\prime} := \inf\{\tilde{t} \sep x+ \tilde{t} y\notin\C\}$, which is equal to $\sup\{t\sep x+ty\in\C \}$, finishes our first claim on existence of a $t^{\prime}$. The second claim on $t^{\prime}=0$ is obvious from $\linhull{\C}$ being a subspace and $x+\epsilon y\in\C$ for some small $\epsilon > 0$ when $x\in\relint{\C}$ and $y\in\linhull{\C}$.

We consider two cases: $y\in\C$ and $y\in\linhull{\C}\setminus\C$. In the first case, $\delta^{\star}$ can be any positive scalar because $\relint{\C}+\C = \relint{\C}$. Now consider the second case $y\in\linhull{\C}\setminus\C$. First we argue the existence of $\delta^{\star}>0$ such that $y+\delta x\in\C$ for all $\delta\ge\delta^{\star}$. Consider $t^{\prime}$ from the first claim in this proof. The value of $t^{\prime}$ is equal to 0 if and only if $x+ty\notin\C$ for all $t>0$. When $x\in\relint{\C}$, there exists some small enough $\epsilon > 0$ such that $x+\epsilon d \in \C$ for all $d\in\aff{\C}$. Taking $d=y$ implies $x+\epsilon y\in\C$, and therefore $t^{\prime} > 0$. Set $\delta^{\star} := 1/t^{\prime}$; we have $\delta^{\star} \in (0,\infty)$ due to $t^{\prime}\in(0,\infty)$. Take any $\delta \ge \delta^{\star}$. Because $1/\delta \in (0,t^{\prime}]$, we have $x+\frac{1}{\delta}y \in\C$. This leads to $\delta(x+ \frac{1}{\delta} y) = y + \delta x \in \C$ because $\C$ is a cone. If $y+\delta x\in\C$ for some $0\le \delta < \delta^{\star}$, then $\frac{1}{\delta}y + x \in\C$, a contradiction to the first claim for $t = 1/\delta$ which is larger than $t^{\prime} := 1/\delta^{\star}$. Hence, we have $y+\delta x\notin\C$ for all $0\le \delta < \delta^{\star}$. It follows then that $y+\delta^{\star}x \in \boundr{\C}$ because $y\notin\C$. Finally, consider $y+\delta x$ for $\delta>\delta^{\star}$. %For this step, we will need the following claim which slightly extends the fact that $\relint{\C}+\relint{\C}\subset\relint{\C}$ due to the relative interior being an open convex cone.
Because $y+\delta x = y+\delta^{\star}x + (\delta-\delta^{\star})x$, and $y+\delta^{\star}x\in\boundr{\C}$ and $(\delta-\delta^{\star})x\in\relint{\C}$ due to $\delta>\delta^{\star}$, $\relint{\C}+\C = \relint{\C}$ gives us $y+\delta x\in\relint{\C}$.
\end{proof}

%Note that $x\in\boundr{\C}$ is not a sufficient condition for $t^{\prime}=0$, e.g., for $\C=\real^{2}_{+}$ and $y = (y_{1},y_{2})$ for some $y_{1} < 0, y_{2}>0$, we have $t^{\prime} = -1/y_{1}$ for $x=(1,0)$ whereas $t^{\prime}=0$ for $x=(0,1)$.

\end{appendices}

\end{document}